%% file: M_approx_R1.tex
\documentclass[11pt,reqno]{amsart}
\input{header_new}

\begin{document}

\title[On the Moreau-Yosida approximation of Posterior distributions]{A Moreau-Yosida  approximation scheme for a class of high-dimensional posterior distributions}\thanks{This work is partially supported by the NSF grant DMS 1228164 and 1513040}

\author{Yves F. Atchad\'e}\thanks{ Y. F. Atchad\'e: University of Michigan, 1085 South University, Ann Arbor,
  48109, MI, United States. {\em E-mail address:} yvesa@umich.edu}



\maketitle

\begin{center} (Feb. 2016, first version May 2015) \end{center}

\begin{abstract}
Exact-sparsity inducing prior distributions in Bayesian analysis typically lead to posterior distributions that are very challenging to handle by standard Markov Chain Monte Carlo  (MCMC) methods, particular in high-dimensional models with large number of parameters. We propose a methodology to derive smooth approximations of such posterior distributions that are, in some cases, easier to handle by standard MCMC methods. The approximation is obtained from the forward-backward approximation of the Moreau-Yosida regularization of the negative log-density. We show that the derived approximation is within $O(\sqrt{\gamma})$ of the true posterior distribution in the $\beta$-metric, where $\gamma>0$ is a user-controlled parameter that defines the approximation.  We illustrate the method with a high-dimensional linear regression model. 
\end{abstract}

\setcounter{secnumdepth}{3}

\section{Introduction}\label{sec:intro}
Successful handling of statistical models with large number of parameters from limited data hinges on the ability to solve efficiently and simultaneously two problems: (a) weeding out non-significant variables, and (b) estimating the effect of the significant variables. 
The concept of sparsity has come to play a fundamental role in this endeavor. In the Bayesian framework, sparsity is naturally built in the prior distribution using spike-and-slab priors (\cite{mitchell:beauchamp:88,george:mcculloch97}), which are mixtures of a point mass at the origin (the spike) and a continuous density (the slab). We will refer to such priors as exact-sparsity inducing priors. A number of recent works have established that these priors, with  carefully chosen slab densities, produce posterior distributions with optimal posterior contraction rates (\cite{castillo:etal:14,atchade:15:b}). However, the flip side of such stellar statistical properties is the fact that these posterior distributions are computationally  difficult to handle, particularly in high-dimensional applications. Deriving tractable and scalable approximations for such distributions is therefore a problem of practical importance.

The most commonly used approach for dealing with posterior distributions from exact-sparsity inducing priors consists in integrating out the regression coefficients (\cite{george:mcculloch97,bottolo:2010}). Recent results by \cite{yang:etal:15} have shown that such an approach indeed scales well with the dimension of the parameter space. However, it holds the limitation that it does not solve both the variable selection and sparse estimation problems jointly. Furthermore, it does not easily extend to non-Gaussian slabs\footnote{For optimal posterior contraction rate, currently available results suggest that one needs slab densities with tails heavier than Gaussian.}, or to non-Gaussian models.  Another solution to dealing with these posterior distributions is to design specialized MCMC samplers, typically using trans-dimensional MCMC techniques such as reversible jump MCMC (\cite{chenetal11}), or the newly proposed shrinkage-thresholding Metropolis adjusted Langevin algorithm (STMaLa; \cite{schrecketal14}). See also \cite{ge:etal:13} for a specialized sampler for blind-deconvolution models. However, these trans-dimensional MCMC samplers currently do not scale well to large problems (\cite{schrecketal14}).

The discussion above suggests that when dealing with exact-sparsity inducing priors in high-dimensional regression problems, scalable approximation of the posterior distribution would be useful. Notice that the Laplace approximation (\cite{tierney:kadane86}), one of the most standard approximation tool in Bayesian computation, cannot be  straightforwardly applied when the dimension of the space is as big as the sample size (\cite{shun:mccullagh95}). Variational Bayes approximations  recently explored by \cite{ormerod:etal:14} form a promising solution, but remained to be fully explored in high-dimensional settings. 

\subsection{Main contribution} We propose a methodology to approximate posterior distributions derived from exact-sparsity inducing priors. An interesting feature of the approximation is that the approximation error is easily controlled by the user. Furthermore, in several important cases, the approximation thus obtained is easily explored by standard Markov Chain Monte Carlo (MCMC) algorithms. The approximation is obtained by taking the forward-backward approximation (closely related to the Moreau-Yosida approximation) of the negative log-density. The Moreau-Yosida approximation is a well-established regularization method in optimization for dealing with non-smooth and constrained problems (\cite{moreau65,bauschke:combettes:2011}). Several recent works have recognized the usefulness of the Moreau-Yosida regularization for Bayesian computation. \cite{pereyra14} noted that a log-concave density can be well approximated by its Moreau-Yosida approximation. However, the framework developed by \cite{pereyra14} does not handle the class of posterior distributions considered here. Another related work is the STMaLa of \cite{schrecketal14} mentioned above,  which implicitly uses the Moreau approximation to design Metropolis-Hastings proposals to sample from posterior distributions with exact-sparsity inducing prior distributions.

If $\check\Pi(\cdot\vert z)$ denotes the posterior distribution of interest on $\rset^d\times \{0,1\}^d$ given data $z$, we write $\check\Pi_\gamma(\cdot\vert z)$ to denote the proposed Moreau-Yosida approximation, where $\gamma>0$ is a user-controlled parameter that defines the quality of the approximation. We derive a general result (Theorem \ref{prop1}) that shows, under some regularity conditions, that 
\begin{equation}\label{eq:intro}
\textsf{d}_{\beta}\left(\check\Pi(\cdot\vert z), \check\Pi_\gamma(\cdot\vert z)\right)\ =O(\sqrt{\gamma}),\end{equation}
where $\textsf{d}_{\beta}$ is the $\beta$-metric that metricizes weak convergence (see Section \ref{sec:notation} for precise definition). One challenge with using the proposed approximation is to find values of $\gamma$ for which $\check\Pi_\gamma$ is close to $\check\Pi$, but not too close so that Markov Chain Monte Carlo samplers with good mixing properties can be easily developed for $\check\Pi_\gamma$. In Theorem \ref{thm2} we propose a choice of $\gamma$ that strikes the aforementioned balance, as we show empirically in the simulation examples. Furthermore, with this particular choice of $\gamma$, we show that the constant in the big $O$ in (\ref{eq:intro}) degrades with the dimension $d$ at most linearly.

We illustrate the method using a linear regression model with a spike-and-slab prior, where the slab is the elastic net density (\cite{li:10}). The example has relevance because the posterior distribution thus defined (actually a special case thereof) is known to contract at the optimal rate (\cite{castillo:etal:14}). Our proposed methodology produces an approximation $\check\Pi_\gamma$ of this posterior distribution, and we develop an efficient Markov Chain Monte Carlo algorithm to sample from $\check\Pi_\gamma$. In this particular example, we show that the Moreau-Yosida approximation actually scales very well with the dimension. More precisely, we derive an upper bound similar to (\ref{eq:intro}) that degrades at most like $\log(d)$ as $d\to\infty$ (see Theorem \ref{thm3}). We illustrate these results in a simulation study which shows that the method performs well, and outperforms STMaLa for high-dimensional problems. A \textsf{Matlab}  implementation  can be obtained from \texttt{http://dept.stat.lsa.umich.edu/$\sim$ yvesa/Research.html}.


The remainder of the paper is organized as follows. We close the introduction with some notation that will be used throughout the paper. In Section \ref{sec:post:dist}, we  first introduce the class of posterior distributions of interest,  followed in Section \ref{sec:approximation} by the basic idea of the Moreau-Yosida approximation. In Section \ref{sec:MYapprox}, we develop how the idea can be applied to approximate the posterior distributions of interest. Section \ref{sec:lm}  details an application to linear regression models. We close the paper with further discussion in Section \ref{sec:comments}. All the proofs are gathered in the Appendix, placed in a supplemental file.

\subsection{Notation}\label{sec:notation}
Throughout the paper, $d\geq 1$ is a given integer and $\rset^d$ denotes the $d$-dimensional Euclidean space equipped with its Borel sigma-algebra, its Euclidean norm $\|\cdot\|$, and inner product $\pscal{\cdot}{\cdot}$.  We also use the norms $\|\theta\|_1\eqdef \sum_{j=1}^d|\theta_j|$, and $\|\theta\|_0$ defined as the number of non-zero components of $\theta$. The Lebesgue measure on $\rset^d$ is written as $\rmd x$ when there is no confusion.

We set $\Delta\eqdef\{0,1\}^d$.  For $\delta\in\Delta$, $\mu_{\delta}$ denote the product measure  on $\rset^d$ defined as $\mu_{\delta}(\rmd \theta)\eqdef \prod_{j=1}^d\nu_{\delta_{j}}(\rmd \theta_{j})$, where  $\nu_{0}(\rmd z)$ is the Dirac mass at $0$, and $\nu_1(\rmd z)$ is the Lebesgue measure on $\rset$. Hence integration with respect to $\mu_\delta$ sets to zero all the components for which $\delta_j=0$, and integrates the remaining components using the standard Lebesgue measure.

For $\theta,\vartheta\in\rset^d$, $\theta\cdot\vartheta$ denotes the component-wise product: $(\theta\cdot\vartheta)_j=\theta_j\vartheta_j$, $1\leq j\leq d$. For $\delta\in\Delta$, we shall write $\theta_\delta$ to denote $\theta\cdot\delta$, and we set 
\[\rset^d_{\delta}\eqdef \{\theta_\delta,\; \theta\in\rset^d\}=\{\theta\in\rset^d:\;\theta_j=0 \mbox{ for } \delta_j=0,\;j=1,\ldots,d\}.\]

We will need ways to evaluate the distance between two probability measures. Let $(\Xset,\dist_\Xset)$ be some arbitrary separable complete metric space equipped with its Borel sigma-algebra. For any two probability measures $\mu_1,\mu_2$ on $\Xset$, the $\beta$-distance between $\mu_1,\mu_2$ is defined as  
\begin{equation}\label{def:beta}
\dist_{\beta}(\mu_1,\mu_2)\eqdef \sup_{\|f\|_{\textsf{BL}}\leq 1} \left|\int_{\Xset} f(x)\mu_1(\rmd x) - \int_{\Xset} f(x)\mu_2(\rmd x)\right|,\end{equation}
where the supremum is taken over all measurable functions $f:\;\Xset\to\rset$ such that $\|f\|_{\textsf{BL}}\eqdef \|f\|_\infty +\|f\|_{\textsf{L}}\leq 1$, where
\[\|f\|_\infty\eqdef \sup_{x\in\Xset}|f(x)|,\;\;\;\mbox{ and }\;\;\|f\|_{\textsf{L}}\eqdef \sup \left\{\frac{|f(x_1)-f(x_2)|}{\dist_{\Xset}(x_1,x_2)},\;x_1,x_2\in\Xset,\;x_1\neq x_2\right\}.\]
It is well-known that this metric metricizes weak convergence (see e.g. \cite{dudley02}~Theorem 11.3.3). If the supremum in (\ref{def:beta}) is replaced by a supremum over all measurable functions $f:\;\Xset\to\rset$ such that $\|f\|_\infty\leq 1$ (resp. $\|f\|_{\textsf{L}}\leq 1$)  one obtains the total variation metric $\textsf{d}_{\textsf{tv}}$ (resp. the Wasserstein metric $\dist_{\textsf{w}}$).

\section{High-dimensional posterior distributions with sparse priors}\label{sec:post:dist}
Let $z$ be a realization of some random variable $Z$ with conditional distribution $f_\theta$, given a parameter $\theta\in\rset^d$. With a prior distribution $\Pi$ on $\theta$, the  posterior distribution for learning $\theta$ is
\begin{equation*}\label{general:post:dist}
\check \Pi(\rmd\theta\vert z)=\frac{f_{\theta}(z)\Pi(\rmd \theta)}{\int_{\rset^d} f_{\theta}(z)\Pi(\rmd \theta)}.
\end{equation*}

Although the prior distribution $\Pi$ can be constructed in a variety of ways, we focus on exact-sparsity inducing  priors (spike-and-slab priors). Such prior distributions have been recently shown to produce posterior distributions with optimal contraction properties (\cite{castillo:etal:14,atchade:15:b}). More specifically, we consider a prior distribution $\Pi$ on $\Delta\times \rset^d$ of the form
\[\Pi(\delta,\rmd\theta) = \pi_\delta \Pi(\rmd \theta\vert\delta),\]
for a discrete distribution $\{\pi_\delta,\delta\in\Delta\}$ on $\Delta$, and a prior $\Pi(\cdot\vert\delta)$  that is built as follows. Given $\delta$, the components of $\theta$ are independent, and for $1\leq j\leq d$,
\begin{equation}\label{basic:prior}
\theta_j\vert \delta\sim \left\{\begin{array}{cc} \textsf{Dirac}(0) & \mbox{ if } \delta_j=0\\ p(\cdot) & \mbox{ if } \delta_j=1\end{array}\right.,\end{equation}
where $\textsf{Dirac}(0)$ is the Dirac measure on $\rset$ with full mass at $0$, and $p(\cdot)$ is a positive density on $\rset$. By the standard data-augmentation trick, we will take the variable $\delta$ as part of the posterior distribution. As defined, the support of $\Pi(\cdot\vert\delta)$ is $\rset^d_\delta=\{\theta\in\rset^d:\;\theta_j=0 \mbox{ for } \delta_j=0,\;1\leq j\leq d\}$, and $\Pi(\cdot\vert\delta)$ has a density with respect to the measure $\mu_{\delta}$ defined in Section \ref{sec:notation}: 
\[\Pi(\rmd\theta\vert \delta) = e^{-P(\theta\vert\delta)}\mu_\delta(\rmd\theta),\;\;\;\mbox{where}\]
\[P(\theta\vert \delta) \eqdef \left\{\begin{array}{ll} -\sum_{j:\;\delta_j=1}\log p(\theta_j) & \mbox{ if } \theta\in\rset^d_\delta\\ +\infty & \mbox{ otherwise }.\end{array}\right.\]
In the above formula, and throughout the paper, we convene that $e^{-\infty}=0$, and $0\times\infty=0$. We also  define  
\[\ell(\theta)\eqdef  -\log f_\theta(z),\;\;\mbox{ and }\;\;h(\theta\vert \delta)\eqdef \ell(\theta) + P(\theta\vert \delta),\;\theta\in\rset^d,\]
so that  the posterior distribution writes
\begin{equation}\label{post:dist}
\check\Pi(\delta,\rmd \theta\vert z)\propto \pi_\delta e^{-h(\theta\vert \delta)}\mu_{\delta}(\rmd\theta).
\end{equation}

Monte Carlo simulation from this posterior distribution can be challenging. The issue is related to the discrete-continuous mixture form of the spike-and-slab prior on $\theta$, which has the effect that any two  distributions $\check \Pi(\delta,\cdot\vert z)$ and $\check \Pi(\delta',\cdot\vert z)$ are mutually singular for $\delta\neq \delta'$. As a result, if direct sampling from the conditional distribution of $\theta\vert\delta,z$ is not possible, then sampling from (\ref{post:dist}) requires the use of trans-dimensional MCMC methods such as reversible jump (\cite{chenetal11}), or STMaLa (\cite{schrecketal14}) which is shown to perform better than reversible jump. However, one issue with STMaLa is that the algorithm has several tuning parameters that are currently poorly understood. Furthermore, as we shall see in the simulations, the mixing of the algorithm degrades significantly for high-dimensional problems, particularly when the signal is weak.

\section{The Moreau-Yosida approximation}\label{sec:approximation}
Our goal in this work is to develop a more tractable approximation to the posterior distribution $\check\Pi$ in (\ref{post:dist}) using the Moreau-Yosida approximation. However to make the ideas easy to follow, we start with some general discussion of the Moreau-Yosida approximation. Let $h:\;\rset^d\to (-\infty,+\infty]$ be a convex, lower semi-continuous function that is not identically $+\infty$, and let $\mu$ be a  sigma-finite measure on $\rset^d$. In the applications, $\mu$ will naturally be taken as the Lebesgue measure on the domain of $h$ (the domain of $h$ is the set of points $x\in\rset^d$ such that $h(x)<\infty$).  Assuming that $Z\eqdef \int_{\rset^d} e^{-h(x)}\mu(\rmd x)<\infty$, we consider the probability measure
\begin{equation}\label{general:meas}
\nu(\rmd x) =\frac{1}{Z} e^{-h(x)}\mu(\rmd x).\end{equation}
To fix the ideas, the reader may think of the case where $h$ is finite everywhere and $\mu$ is the Lebesgue measure on $\rset^d$. In that case $\nu$ is the probability distribution on $\rset^d$ with  density $(1/Z)e^{-h(x)}$. However our main interest is in the posterior distribution (\ref{post:dist}) for which the slightly more general setting is needed. 

Suppose that we are interested in drawing samples from $\nu$. The lack of smoothness of $h$, and the possibly complicated geometry of the support of $\nu$ can create difficulties for standard MCMC algorithms. A smooth approximation of $\nu$ can be formed from the Moreau-Yosida approximation of $h$  defined for $\gamma>0$ as
\[\tilde h_\gamma(x)=\min_{u\in\rset^d}\left[h(u) +\frac{1}{2\gamma}\|u-x\|^2\right],\;\;x\in\rset^d.\]
Under the assumptions imposed on $h$ above, the function $\tilde h_\gamma$ is known to be well-defined and finite everywhere. It is also convex, continuously differentiable with a Lipschitz gradient, and $\tilde h_\gamma(x)\uparrow h(x)$, as $\gamma\downarrow 0$, for all $x\in\rset^d$. All these properties are well-known and can be found in \cite{bauschke:combettes:2011} (Chapter 12). Assuming that $\tilde Z_\gamma\eqdef \int_{\rset^d} e^{-\tilde h_\gamma(x)}\rmd x<\infty$, it seems natural to consider the probability measure
\[\tilde \nu_\gamma(\rmd x) =\frac{1}{\tilde Z_\gamma} e^{-\tilde h_\gamma(x)}\rmd x,\]
as an approximation of $\nu$.  To the best of our knowledge, the approximation $\tilde \nu_\gamma$ was first considered by \cite{pereyra14}, for a probability distribution $\nu$ for which  $h$ is finite everywhere and $\mu$ is the Lebesgue measure on $\rset^d$. And we refer the reader to that paper for a good discussion of the basic properties of $\tilde\nu_\gamma$, and how well it approximates $\nu$. In particular \cite{pereyra14} showed that the smoothness of $\tilde h_\gamma$ can be exploited to derive efficient gradient-based MCMC samplers for $\nu$. An important limitation of the Moreau-Yosida approximation is that it is typically not available in closed form, and its computation leads to a $d$-dimensional, possibly complicated optimization problem. 

In many problems the function $h$ takes the particular form 
\[h(x) =\ell(x) + P(x),\;\;x\in\rset^d\]
where $\ell$ is convex, finite everywhere and twice continuously differentiable, and $P$ is convex, not identically $+\infty$ and lower semi-continuous. In such cases, one can approximate $\ell$ around a given point $x$ by its linear approximation $u\mapsto \ell(x) +\pscal{\nabla\ell(x)}{u-x}$, where $\nabla\ell(x)$ denote the gradient of $\ell$ at $x$. This approximation leads to the so-called forward-backward approximation of $h$, defined for $\gamma>0$ as
\begin{eqnarray}\label{def:nu_gam}
 h_\gamma(x) & \eqdef & \min_{u\in\rset^d}\left[\ell(x) +\pscal{\nabla \ell(x)}{u-x} +P(u) +\frac{1}{2\gamma}\|u-x\|^2\right],\;\;x\in\rset^d \nonumber\\
 & =&  \ell(x) + -\frac{\gamma}{2}\|\nabla\ell(x)\|^2 + \min_{u\in\rset^d}\left[P(u)+\frac{1}{2\gamma}\|u-x+\gamma\nabla\ell(x)\|^2\right].\end{eqnarray}
Under the assumptions imposed on $\ell$ and $P$ above, the function $h_\gamma$ is finite everywhere, continuously differentiable, and $h_\gamma\leq h$. These properties can be found in  \cite{patrinos:etal:2014}~Theorem 2.2, but are easy to derive. For instance, the differentiability follows from the expression (\ref{def:nu_gam}), the twice differentiability of $\ell$,  and the differentiability of the Moreau-Yosida approximation of $P$. Notice however that $h_\gamma$ is no longer convex in general. Assuming that $Z_\gamma\eqdef \int_{\rset^d} e^{-h_\gamma(x)}\rmd x<\infty$ it seems also natural to consider the resulting approximation of $\nu$ defined as 
\[\nu_\gamma(\rmd x) =\frac{1}{Z_\gamma} e^{-h_\gamma(x)}\rmd x.\]
The main advantage of $h_\gamma$ over $\tilde h_\gamma$ is that in many problems of interest $h_\gamma$ is available in closed form, whereas $\tilde h_\gamma$ is not. Furthermore, if the function $P$ is separable, then the computation of $h_\gamma$ leads to $d$ separate one-dimensional optimization problems. However, the price to pay for the computational convenience is that $h_\gamma$ may not be convex, and it is a less accurate approximation of $h$. Indeed, by the convexity of $\ell$, we have $\ell(u)\geq \ell(x) +\pscal{\nabla \ell(x)}{u-x}$ for all $u\in\rset^d$. Hence $h_\gamma(x)\leq \tilde h_\gamma(x)\leq h(x)$ for all $x\in\rset^d$. As we will see, the convergence $h_\gamma(x)\uparrow h(x)$, as $\gamma\downarrow 0$, for all $x\in\rset^d$, still holds. Figure \ref{Fig:0} gives an illustrative example of the differences between $h_\gamma$ and $\tilde h_\gamma$ and how both functions approximate $h$. 

\begin{figure}[h!]
\centering
\scalebox{0.35}{\includegraphics{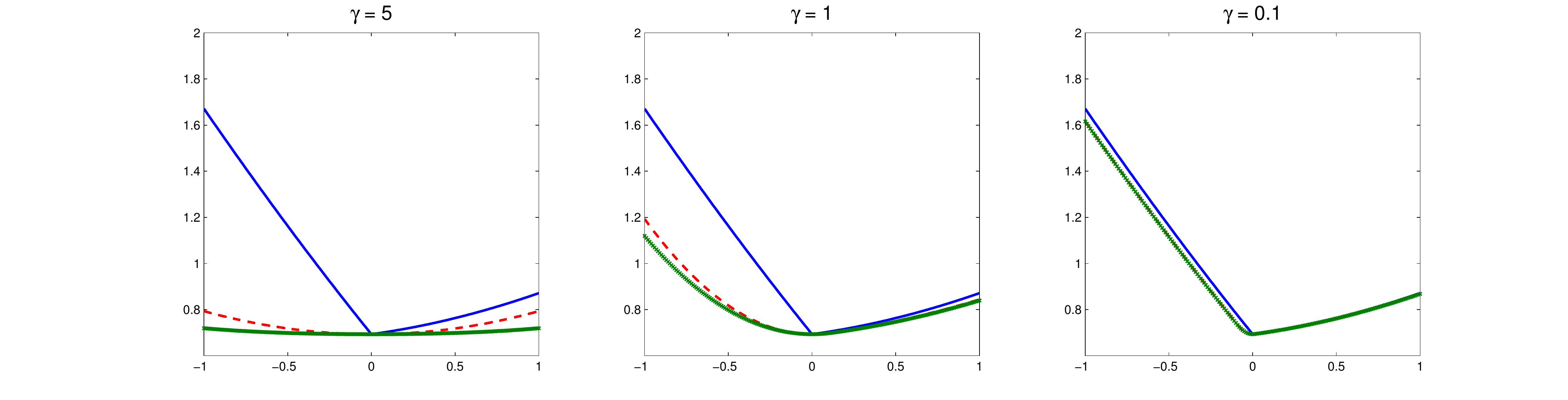}}
\caption{{\small{Figure showing the function $h(x) = -ax +\log(1+e^{ax}) +b|x|$ for $a=0.8$, $b=0.5$ (blue/solid line), and the approximations $h_\gamma$ and $\tilde h_\gamma$ ($h_\gamma\leq\tilde h_\gamma$), for $\gamma\in\{5,1,0.1\}$. For $\gamma=0.1$, the curves of $h_\gamma$ and $\tilde h_\gamma$ are almost undistinguishable on the figure.}}}\label{Fig:0}
\end{figure}

Since $h_\gamma$ converges pointwise to $h$ as $\gamma\downarrow 0$, it seems natural to expect that $\nu_\gamma$ approaches $\nu$ for small $\gamma$. 
If the function $h$ is finite everywhere, one can easily show (see Proposition \ref{prop0} below) that indeed, $\nu_\gamma$ converges to $\nu$ in the total variation metric, as $\gamma\downarrow 0$. However this result is no longer true when the domain of $h$ has zero $\rset^d$-Lebesgue measure. In this latter case, we will show that the convergence of $\nu_\gamma$ occurs only weakly, or in the Wasserstein metric.

\begin{proposition}\label{prop0}
Suppose $\mu$ in (\ref{general:meas}) is the Lebesgue measure on $\rset^d$, $h=\ell+P$ is convex, finite everywhere, and $h_\gamma(x)\uparrow h(x)$ for all $x\in\rset^d$. Suppose also that there exists $\gamma_0>0$ such that $Z_{\gamma_0}<\infty$. Then for all $\gamma\in(0,\gamma_0]$, $\nu_\gamma$ is well-defined,
and
\[\dist_{\tv}(\nu_\gamma, \nu) \leq 2\left(1-\frac{Z}{Z_\gamma}\right)\downarrow 0,\;\;\mbox{ as }\;\gamma\downarrow 0.\]
\end{proposition}
\begin{proof}
See the Appendix.
\end{proof}

\begin{remark}
\begin{enumerate}
\item Notice that Proposition \ref{prop0} can also be applied to $\tilde\nu_\gamma$ by taking $\ell\equiv 0$.
\item We show in Lemma 2 in the Appendix that if $\ell$ is finite everywhere and differentiable, and $P$ is finite everywhere, convex with a nonempty subdifferential at $x$ for all $x\in\rset^d$, then $h_\gamma\uparrow h$,  as required in the proposition.
\end{enumerate}
\end{remark}

\medskip
If the domain of $h$ has $\rset^d$-Lebesgue measure $0$, then $\nu$ and $\nu_\gamma$ are then automatically mutually singular and Proposition \ref{prop0} cannot hold. The following toy example illustrates this case.

\begin{example}\label{ex:1}
Suppose that we take $\rset^d=\rset$, $\ell\equiv 0$, and we take  $P(x)=0$ is $x=0$, and $P(x)=+\infty$ if $x\neq 0$. In that case $e^{-h(x)}=1$ if $x=0$, and $e^{-h(x)}=0$ if $x\neq 0$. Let $\mu=\delta_0$ be the point mass probability measure at $0$. Hence $\nu=\delta_0$. For $\gamma>0$, $h_\gamma(x)=\tilde h_\gamma(x) =x^2/(2\gamma)$, $x\in\rset$. Hence $\nu_\gamma$ is the normal distribution $\textbf{N}(0,\gamma)$. It follows that $\textsf{d}_{\tv}(\nu_\gamma, \nu) = 2$, for all $\gamma>0$. But for any Lipschitz function $f:\;\rset\to\rset$ with Lipschitz constant $1$, 
\[\left|\nu_\gamma(f)-\nu(f)\right| =\left|\nu_\gamma(f)-f(0)\right|\leq \PE(|Z_\gamma|)= \sqrt{\frac{2\gamma}{\pi}},\]
where $Z_\gamma \sim \textbf{N}(0,\gamma)$. By taking $f=|\cdot|$, it can be easily seen that $\textsf{d}_{\textsf{w}}(\nu_\gamma,\nu) = \sqrt{\frac{2\gamma}{\pi}}$. Hence $\nu_\gamma$ converges in the Wasserstein metric to $\nu$, but not in total variation. And the convergence rate is $O(\sqrt{\gamma})$.
\end{example}

\begin{remark}
The fact that we only have convergence in the Wasserstein metric has practical implications. It implies that one needs to be cautious about what probability $\nu(A)$ can be well approximated by $\nu_\gamma(A)$. For instance, in Example \ref{ex:1}, if $A$ is of the form $(a,0]$ or $[0,b]$, then  $\nu(A)=1$, whereas $\lim_{\gamma\downarrow 0}\nu_\gamma(A)= 0$. 
\end{remark}

In the next section we will use the approximating measure $\nu_\gamma$ introduced above to approximate the posterior distribution (\ref{post:dist}). We will see that the situation is similar to the one in Example \ref{ex:1}, and as in that example we will show that the approximation converges weakly to the posterior distribution $\check \Pi$.

\section{The Moreau-Yosida approximation of the posterior distribution (\ref{post:dist})}\label{sec:MYapprox}
In this section, we return to the posterior distribution (\ref{post:dist}) defined in Section \ref{sec:post:dist}. And we make the following assumptions on the functions $\ell$ and $P$.

\begin{assumption}\label{A1}
\begin{enumerate}
\item The function $\theta\mapsto \ell(\theta)$ is finite everywhere, convex, and twice continuously differentiable. 
 \item For all $\delta\in\Delta$, the function $\theta\mapsto P(\theta\vert\delta)$ is convex, lower semi-continuous, not identically $+\infty$, and admits a sub-gradient $g(\theta\vert\delta)$ at $\theta$, for all $\theta\in\rset^d_\delta$. \end{enumerate}
\end{assumption}

\begin{remark}
\begin{enumerate}
\item The convexity assumption on $\ell$ is fundamental and delineates the type of problems to which the proposed approximation could be easily applied. Extension beyond this set up is possible, but will require fundamentally different techniques.  
\item The convexity of $P(\cdot\vert\delta)$ boils down to the log-concavity of the density $p$ in the prior (\ref{basic:prior}). Most of the sparsity promoting prior densities used in practice are log-concave. 
\end{enumerate}
\end{remark}

Given $\delta\in\Delta$, we consider the forward-backward approximation of $h(\cdot\vert\delta)$ defined as
\begin{equation}\label{def:hng1}
h_{\gamma}(\theta\vert\delta)\eqdef\min_{u\in\rset^d}\left[ \ell(\theta) +\pscal{\nabla \ell(\theta)}{u-\theta} +P(u\vert\delta) +\frac{1}{2\gamma}\|u-\theta\|^2\right],\;\theta\in\rset^d,\end{equation}
for some parameter $\gamma>0$. Using $h_{\gamma}$, we propose to approximate the posterior distribution $\check\Pi$ in (\ref{post:dist}) by
\begin{equation}\label{quasi:post:dist}
\check\Pi_{\gamma}(\delta, \rmd \theta\vert z) \propto \pi_\delta \left(2\pi\gamma\right)^{\frac{\|\delta\|_0}{2}}  e^{-h_{\gamma}(\theta\vert \delta)}\rmd \theta,
\end{equation}
that we call the Moreau-Yosida approximation of $\check \Pi$, although as we have seen above, (\ref{def:hng1}) is only the forward-backward approximation of  $h(\cdot\vert\delta)$. In the expression (\ref{quasi:post:dist}), $\pi$ denotes the irrational number. The function $h_{\gamma}(\cdot\vert\delta)$ is available in closed form whenever the Moreau-Yosida approximation of $P(\cdot\vert\delta)$ has a closed form expression. More specifically, for  $\delta\in\Delta$, and for $\gamma>0$, we define
\begin{equation}\label{Moreau-Yosida:P}
P_{\gamma}(\theta\vert \delta) \eqdef \min_{u\in\rset^d}\left[ P(u\vert\delta) +\frac{1}{2\gamma}\|u-\theta\|^2\right],\;\;\theta\in\rset^d\end{equation}
 the Moreau-Yosida approximation of $P$, and its associated proximal map
\[\Prox_\gamma(\theta\vert \delta) \eqdef \textsf{Argmin }_{u\in\rset^d}\left[ P(u\vert\delta) +\frac{1}{2\gamma}\|u-\theta\|^2\right],\;\;\theta\in\rset^d.\]
From the definition of $P_\gamma$ and $\Prox_\gamma$, we see  that $h_{\gamma}$ can be alternatively written as
\begin{eqnarray}\label{def:hng2}
h_{\gamma}(\theta\vert\delta)&=& \ell(\theta) -\frac{\gamma}{2}\|\nabla \ell(\theta)\|^2 +P_{\gamma}\left(\theta-\gamma \nabla \ell(\theta)\vert \delta\right)\\
&=& \ell(\theta) +\pscal{\nabla\ell(\theta)}{J_\gamma(\theta\vert\delta)-\theta} +P(J_\gamma(\theta\vert \delta)\vert\delta) \nonumber\\
\label{def:hng3}&&+\frac{1}{2\gamma}\|J_\gamma(\theta\vert\delta)-\theta\|^2,\end{eqnarray}
where
\[J_\gamma(\theta\vert\ \delta)\eqdef \Prox_\gamma\left(\theta-\gamma \nabla\ell(\theta)\vert \delta\right).\]
For $\gamma>0$, $\theta\in\rset^d$, let $\textsf{s}_\gamma(\theta)\in\rset^d$ be such that 
\[(\textsf{s}_\gamma(\theta))_j\eqdef \textsf{Argmin }_{u\in\rset}\left[-\log p(u) +\frac{1}{2\gamma}(u-\theta_j)^2\right],\;\;1\leq j\leq d.\]
Then it is easy to check that $\Prox_\gamma(\theta\vert\delta) =\delta\cdot \textsf{s}_\gamma(\theta)$.  Hence by Equation (\ref{def:hng3}), we see that $h_\gamma(\cdot\vert\delta)$ is computationally tractable if the map $\textsf{s}_\gamma$ (the proximal map of the negative log-prior) is easy to compute. Although this limits the applicability of the method, there a several priors commonly used for which this holds, including the Laplace prior and the elastic-net prior given respectively by
\[p(u) \propto e^{-\lambda |u|}, \;\;\mbox{ and }\;\; p(u)\propto \exp\left(-\alpha\lambda_1|u| -(1-\alpha)\lambda_2 \frac{u^2}{2}\right),\]
as well as the generalized double Pareto of \cite{armaganetal13b}, and the (improper) prior distribution that arises from the MCP of \cite{zhang:10}, given respectively by
\[p(u) = \frac{1}{2\lambda}\left(1+\frac{|u|}{\alpha\lambda}\right)^{-(\alpha+1)},\;\;\mbox{ and }\;\;p(u) = \exp\left(-\lambda \int_0^{|u|}\left(1-\frac{t}{\alpha\lambda}\right)_+ \rmd t\right).\]

\subsection{Connection with spike-and-slab priors}
The proposed approximation $\check\Pi_\gamma$ is closely related to the distribution obtained by replacing all the Dirac mass in (\ref{basic:prior}) by independent Gaussian distributions $\textbf{N}(0,\gamma)$, $\gamma>0$. More precisely, let $\tilde\Pi_\gamma$ denote the posterior distribution of $(\delta,\theta)$ in the following model. 
\begin{equation}\label{model:2}
\delta\sim \{\pi_\delta\},\;\;\;\;\theta_j\vert \delta\sim \left\{\begin{array}{cc} \textbf{N}(0,\gamma) & \mbox{ if } \delta_j=0\\ p(\cdot) & \mbox{ if } \delta_j=1\end{array}\right., 1\leq j\leq d,\;\mbox{ and }\; Z\vert\delta,\theta\sim f_{\theta_\delta}.\end{equation}
Notice that in (\ref{model:2}) given $(\delta,\theta)$, we draw $Z$ from $f_{\theta_\delta}$, with a sparse parameter $\theta_\delta$. The posterior distribution thus defined is
\begin{equation}\label{tilde:pi}
\tilde\Pi_\gamma(\delta,\rmd\theta\vert z)\propto \pi_\delta\left(\frac{1}{2\pi\gamma}\right)^{\frac{d-\|\delta\|_1}{2}} e^{-\frac{1}{2\gamma}\|\theta-\theta_\delta\|^2} e^{-h(\theta_\delta\vert\delta)}\rmd\theta.
\end{equation}
The distribution $\tilde\Pi_\gamma$ in turn, is closely related to another posterior distribution commonly used in practice and obtained from the following model:
\begin{equation}\label{model:3}
\delta\sim \{\pi_\delta\},\;\;\;\;\theta_j\vert \delta\sim \left\{\begin{array}{cc} \textbf{N}(0,\gamma) & \mbox{ if } \delta_j=0\\ p(\cdot) & \mbox{ if } \delta_j=1\end{array}\right., 1\leq j\leq d,\;\mbox{ and }\; Z\vert\delta,\theta\sim f_{\theta},\end{equation}
for some constant $\gamma>0$. Here given $(\delta,\theta)$, we draw $Z$ from $f_{\theta}$. Model (\ref{model:3}) is widely used in practice as a more tractable alternative to the point-mass spike-and-slab (\cite{george:mcculloch97,ishwaran:rao:2005,rockova:george:14,narisetti:he:14}). Clearly, to the extend that the point-mass spike-and-slab is the gold-standard, the model in (\ref{model:2}) is preferable to the one in (\ref{model:3}). The Moreau-Yosida approximation proposed in this paper can be viewed as a very close approximation to $\tilde\Pi_\gamma$, as we show that $\textsf{d}_\tv(\check\Pi_\gamma,\tilde\Pi_\gamma)=O(\gamma)$ (see Lemma 3 in the Appendix, and (\ref{thm2:rate})), where $\textsf{d}_\tv$ denotes the total variation metric. The interest of our method  then comes from the fact that sampling from $\check\Pi_\gamma$ is much easier than sampling from $\tilde\Pi_\gamma$. Indeed, notice that in $\tilde\Pi_\gamma$, the parameter $\delta$ appears also in the likelihood function $\ell(\theta_\delta)$. As a result, both conditional distributions $\delta\vert z,\theta$ and $\theta\vert z,\delta$ are typically intractable and require MCMC algorithms.  Whereas in $\check\Pi_\gamma$, given $(z,\theta)$, the components of $\delta$ are independent Bernoulli random variables.

\subsection{Approximation bounds}
We will now derive a result that bounds the $\beta$-distance between $\check\Pi_{\gamma}$ and $\check\Pi$.  We  define
\begin{equation}\label{varrho}
\varrho_\gamma(z)\eqdef \log\int e^{r_\gamma(\delta,\theta)}\tilde\Pi_\gamma(\rmd\delta,\rmd\theta\vert z),
\end{equation}
where 
\[r_\gamma(\delta,\theta) \eqdef \pscal{\nabla\ell(\theta)-\nabla\ell(\theta_\delta)}{\theta-\theta_\delta)} + \frac{\gamma}{2}\|\delta\cdot\nabla\ell(\theta) +\delta\cdot g\left(\theta_ \delta\vert\delta\right)\|^2.\]
For simplicity, we shall omit the dependence of $r_\gamma(\delta,\theta)$ on $z$ (same with $\ell(\theta)$ and $\nabla\ell(\theta))$. We note that by the convexity of $\ell$, $r_\gamma(\delta,\theta)\geq 0$. Hence $\varrho_\gamma(z)\geq 0$.

\begin{theorem}\label{thm1}
Assume H\ref{A1}, for some fixed data $z$. Suppose that there exists $\gamma_0>0$ such that $\check\Pi_{\gamma_0}(\cdot\vert z)$ is well-defined. Then for all $\gamma\in (0,\gamma_0]$, $\check\Pi_{\gamma}(\cdot\vert z)$ is well-defined  and
\begin{equation}\label{thm1:rate}
\dist_{\beta}\left(\check\Pi_{\gamma}(\cdot\vert z),\check \Pi(\cdot\vert z)\right)\leq \sqrt{\gamma d} + 2\left(1-e^{-\varrho_\gamma(z)}\right).\end{equation}
\end{theorem}
\begin{proof}
See the Appendix.
\end{proof}

Notice that $1-e^{-x}\leq x$ for all $x\geq 0$. Therefore, the convergence to zero of $\dist_{\beta}\left(\check\Pi_{\gamma}(\cdot\vert z),\check \Pi(\cdot\vert z)\right)$ would follow if the term $\varrho_\gamma(z)$ converges to $0$ as $\gamma\to 0$. In the next result we impose some additional assumptions, which, together with H\ref{A1} guarantee that $\varrho_\gamma(z)$ converges to zero. In the process we  derive an explicit bound on the convergence rate which can be used  to develop guidelines for choosing $\gamma$. We make the following assumption.

\begin{assumption}\label{A2}
\begin{enumerate}
\item There exists $L_1<\infty$ such that,
\begin{equation}\label{Lip:cond}
\|\nabla\ell(\theta_1)-\nabla\ell(\theta_2)\|\leq L_1 \|\theta_1-\theta_2\|,\;\;\theta_1,\;\theta_2\in\rset^d.\end{equation}
\item There exists $L_2<\infty$ such that
\begin{equation}\label{ell:cond1}
\|\delta\cdot\nabla\ell(\theta)\|^2\leq 2L_2 \ell(\theta),\;\;\delta\in\Delta,\;\theta\in\rset^d_\delta.\end{equation}
\item For all $\delta\in\Delta$, there exists $c(\delta)<\infty$, such that
\begin{equation}\label{ell:cond2}
\|\delta\cdot g(\theta\vert\delta)\|^2\leq c(\delta) +2L_2 P(\theta\vert\delta),\;\;\theta\in\rset^d_\delta.\end{equation}
\end{enumerate}
\end{assumption}

\begin{remark}
H\ref{A2}-(1) is a standard Lipschitz assumption. H\ref{A2}-(2) and H\ref{A2}-(3) essentially requires both functions $\ell$ and $P(\cdot\vert\delta)$ to grow like $O(\|\theta\|^2)$, or $o(\|\theta\|^2)$,  as $\|\theta\|\to\infty$.
\end{remark}
\begin{theorem}\label{thm2}
Assume H\ref{A1}-H\ref{A2}, for some fixed data $z$, and suppose $\gamma>0$ is such that $4\gamma \max(L_1,L_2)\leq 1$. Then $\check\Pi_{\gamma}(\cdot\vert z)$ is well-defined  and
\begin{equation*}
\dist_{\beta}\left(\check\Pi_{\gamma}(\cdot\vert z),\check \Pi(\cdot\vert z)\right)\leq \sqrt{\gamma d} + 2\left(1-e^{-\varrho_\gamma(z)}\right),\end{equation*}
where
\begin{equation}\label{thm2:rate}
\varrho_\gamma(z)\leq 3\gamma \left[\frac{1}{2} \max_{\delta\in\Delta} c(\delta) + d(L_1 +2 L_2) + L_2\mathcal{R}(z)\right],
\end{equation}
where $\mathcal{R}(z)\eqdef\max_{\delta\in\Delta}\inf_{\theta\in\rset^d}\left[\ell(\theta_\delta) + P(\theta_\delta\vert \delta)\right]\leq \max_\delta[\ell(0) + P(0\vert\delta)]$.
\end{theorem}
\begin{proof}
See the Appendix.
\end{proof}

Since $1-e^{-x}\leq x$, for all $x\geq 0$, Theorem \ref{thm2} shows that as $\gamma\to 0$, the Moreau-Yosida approximation $\check\Pi_{\gamma}(\cdot\vert z)$ approaches $\check\Pi(\cdot\vert z)$ at the rate of $\sqrt{\gamma}$, under H\ref{A1} and H\ref{A2}. The rate is optimal as Example \ref{ex:1} shows. Theorem \ref{thm2} also provides some guidelines for choosing $\gamma$, as it suggests that one can choose $\gamma$ as 
\begin{equation}\label{choice:gamma:1}
\gamma = \frac{\gamma_0}{\max(L_1,L_2)},\;\;\mbox{ with }\;\; 0<\gamma_0\leq \frac{1}{4}.\end{equation}
The bound in (\ref{thm2:rate}) seems to suggest that the quality of the approximation resulting from choosing $\gamma$ as in (\ref{choice:gamma:1}) degrades only linearly with the dimension $d$, as $d$  increases\footnote{Indeed, the term $\mathcal{R}(z)$ always satisfies $\mathcal{R}(z)\leq \max_\delta[\ell(0) + P(0\vert\delta)]$, and this latter expression typically does not grow with $d$.}. However, it is important to realize that the bound in (\ref{thm2:rate}) is most likely not tight, and the dependence of $\varrho_\gamma(z)$ on $d$ could be even better than $O(d)$ (see the linear regression example below).  In general we cautious against the use of too small values of $\gamma$, since choosing $\check\Pi_\gamma$ very close to $\check\Pi$ limits the ability to construct good MCMC sampler to explore $\check\Pi_\gamma$.

\section{Application to Bayesian linear regression with sparse priors}\label{sec:lm}

As an application we consider a high-dimensional linear regression problem, with dependent variable $z\in\rset^n$, and design matrix $X\in\rset^{n\times d}$. The variance term $\sigma^2$ is assumed known. The negative-log-likelihood function $\ell$ for this problem can be taken as
\[\ell(\theta) =\frac{1}{2\sigma^2}\|z-X\theta\|^2,\;\;\theta\in\rset^d. \]
We will set up the prior distribution of $\theta$ using $\delta\in\Delta$, and using an auxiliary variable $\phi\eqdef (\piq,\lambda_1,\lambda_2)$, where $\piq\in(0,1)$ is a sparsity parameter, and $\lambda_1>0$, $\lambda_2>0$ are regularization parameters.  Given $\phi$, we assume that the components of $\delta$ are independent and identically distributed, with distribution $\textbf{Ber}(\piq)$. Hence $\pi_\delta = \piq^{\|\delta\|_0}(1-\piq)^{d-\|\delta\|_0}$. Given $\phi$ and $\delta$, the components of $\theta$ are independent, and for $1\leq j\leq d$,
\[\theta_j\vert \delta,\phi\sim \left\{\begin{array}{cc} \textsf{Dirac}(0) & \mbox{ if } \delta_j=0\\ \textsf{EN}\left(\frac{\lambda_1}{\sigma^2},\frac{\lambda_2}{\sigma^2}\right) & \mbox{ if } \delta_j=1\end{array}\right.,\]
where $\textsf{Dirac}(0)$ is the Dirac measure on $\rset$ with full mass at $0$, and $\textsf{EN}(\lambda_1/\sigma^2,\lambda_2/\sigma^2)$ is the (elastic-net) distribution  with density given by
\begin{equation}\label{prior:p}
\frac{1}{Z(\phi)} \exp\left(-\alpha\frac{\lambda_1}{\sigma^2}|x| - (1-\alpha)\frac{\lambda_2}{2\sigma^2}x^2\right),\;\;x\in\rset,\end{equation}
for a parameter $\alpha\in[0,1]$, assumed known. The normalizing constant $Z(\phi)$ can be written as 
\begin{equation*}\label{eq:Z}
Z(\phi) = \left\{\begin{array}{ll}\sigma \sqrt{\frac{2\pi}{(1-\alpha)\lambda_2}} \textsf{erfcx}\left(\frac{\alpha\lambda_1}{\sigma\sqrt{2(1-\alpha)\lambda_2}}\right) & \mbox{ if } \alpha\in[0, 1)\\ \frac{2\sigma^2}{\lambda_1} & \mbox{ if } \alpha=1\end{array}\right.,\end{equation*}
where $\textsf{erfcx}(x)$ is the scaled complementary error function, which can be written as $\textsf{erfcx}(x) = 2e^{x^2}\Phi(-\sqrt{2}x)$, where $\Phi$ is the cdf of standard normal distribution. 
The prior density (\ref{prior:p}) is a reparametrization of the elastic-net (\cite{zou:hastie:05}) prior used by \cite{li:10}. Notice that $\alpha=1$ makes $\lambda_2$ inactive, and setting $\alpha=0$ makes $\lambda_1$ inactive. 

Given the prior specified above, the function $P$ becomes 
\[P(\theta\vert \delta) = \|\delta\|_1\log Z(\phi) +\frac{\alpha\lambda_1}{\sigma^2}\|\theta_\delta\|_1 + \frac{(1-\alpha)\lambda_2}{2\sigma^2}\|\theta_\delta\|^2 + \iota_{\rset^d_\delta}(\theta),\;\;\theta\in\rset^d,\]
where $\iota_{\rset^d_\delta}(\theta) = 0$ if $\theta\in\rset^d_{\delta}$, and $\iota_{\rset^d_\delta}(\theta)=+\infty$ otherwise. We recall that $\rset^d_\delta=\{\theta\in\rset^d:\; \theta_j=0,\mbox{ if } \delta_j=0,\;j=1,\ldots,d\}$. With $h(\theta\vert\delta) = \ell(\theta) + P(\theta\vert\delta)$,  the posterior distribution of $(\delta,\theta)$ is
\[
\check \Pi\left(\delta,\rmd \theta\vert z\right)\propto \pi_\delta e^{-h(\theta\vert\delta)}\mu_{\delta}(\rmd\theta).
\]
With the elastic net prior (\ref{prior:p}), the proximal function $\Prox_\gamma(\theta\vert\delta)$ is easy to compute. For $x\in\rset$, define $\textsf{sign}(x)$ as $1$ is $x>0$, $-1$ if $x<0$ and $0$ if $x=0$. For $\gamma>0$, let $\textsf{s}_\gamma(\theta)\in\rset^d$ denotes the vector whose $j$-th component is given by 
\begin{equation}\label{lasso:op}
(\textsf{s}_\gamma(\theta))_{j}=\frac{\textsf{sign}(\theta_{j})\left(|\theta_{j}| -\alpha\gamma\frac{\lambda_1}{\sigma^2}\right)_+}{1+\gamma\frac{\lambda_2}{\sigma^2}(1-\alpha)}.\end{equation}
It is easy to show that
\[ \Prox_\gamma(\theta\vert \delta) = \delta\cdot \textsf{s}_\gamma(\theta),\]
where $\theta_1\cdot\theta_2$ denotes the component-wise product. From (\ref{quasi:post:dist}), it follows that the Moreau-Yosida approximation $\check\Pi_{\gamma}$ of $\check\Pi$  has a density $\check\pi_{\gamma}$ given by
\begin{equation}\label{post:approx:lm}
\check \pi_{\gamma}\left(\delta,\theta\vert z\right)\propto \pi_\delta \left(2\pi\gamma\right)^{\frac{\|\delta\|_0}{2}}  e^{-h_{\gamma}(\theta\vert\delta)},\end{equation}
where $h_\gamma(\cdot\vert \delta)$ is given by (\ref{def:hng3}). In the next result, we show that H\ref{A1} and H\ref{A2} hold for this problem, and Theorem \ref{thm2} applies. For a matrix $A$, let $\lambda_{\textsf{max}}(A)$ denote its largest eigenvalue.
\begin{corollary}\label{prop1}
Suppose that $(1-\alpha)\lambda_2\leq \lambda_{\textsf{max}}(X'X)$, and suppose that $\gamma>0$ satisfies 
\begin{equation}\label{choice:gamma:2}
\frac{4\gamma}{\sigma^2}\lambda_{\textsf{max}}(X'X)\leq 1. \end{equation}
Then for all $z\in\rset^n$, $\check \Pi_{\gamma}(\cdot\vert z)$ is a well-defined probability measure on $\Delta\times\rset^d$, and
\[\textsf{d}_\beta\left(\check\Pi_{\gamma}(\cdot\vert z),\check\Pi(\cdot\vert z)\right) \leq \sqrt{\gamma d} +2\left(1-e^{-\varrho_\gamma(z)}\right),\]
where $\varrho_\gamma(z)$ satisfies
\begin{equation}\label{prop1:rate}
\varrho_\gamma(z) \leq \frac{3\gamma}{2} \left(\frac{\alpha\lambda_1}{\sigma^2}\right)^2 d + \frac{3\gamma}{\sigma^2}\lambda_{\textsf{max}}(X'X)\left(3d +\frac{\|z\|^2}{2\sigma^2}\right).
\end{equation}
\end{corollary}
\begin{proof}
See the Appendix.
\end{proof} 

As in the general case above, (\ref{choice:gamma:2}) suggests choosing
\begin{equation}\label{eq:choice:gamma:2}
\gamma=\frac{\gamma_0\sigma^2}{\lambda_{\textsf{max}}(X'X)},\;\;\;\gamma_0\in (0,1/4].\end{equation}
And with this choice,  the bound in (\ref{prop1:rate}) deteriorates only linearly in $d$. In fact, (\ref{prop1:rate}) is a worst case analysis and better bounds can be derived if one takes into account the sampling distribution of the data $z$.  We prove one such result below.

We shall take the frequentist viewpoint and assume that the observed data $z$ is a realization of $Z$ where
\begin{equation}\label{dist:assump}
Z =X\theta_\star +\epsilon,\;\;\;\;\; \mbox{ where}\;\;\;\; \epsilon\sim\textbf{N}(0,\sigma^2 I_n),\end{equation}
for a sparse unknown vector $\theta_\star\in\rset^d$. In the Bayesian as in the frequentist framework, the recovery of $\theta_\star$ when $d>n$ depends on the positiveness of some restricted and sparse eigenvalues of $X'X$. We define these quantities next. Let $\delta_\star\in\Delta$, denote the sparsity structure of $\theta_\star$. That is, $\delta_{\star,j}=1$ if and only if $|\theta_{\star,j}|>0$. We set $s_\star\eqdef \|\theta_\star\|_0$, the number of non-zero components of $\theta_\star$. We define 
\begin{equation*}\label{def:phi}
\underline{\kappa}\eqdef  \inf \left\{\frac{\theta'(X'X)\theta}{n\|\theta\|^2}:\;\theta\neq 0,\;\|\theta-\theta_{\delta_\star}\|_1\leq 7\|\theta_{\delta_\star}\|_1\right\},
\end{equation*}
and $s\in\{1,\ldots,d\}$,  we define
\[
\bar\kappa(s)\eqdef \sup \left\{\frac{\theta'(X'X)\theta}{n\|\theta\|^2}:\;1\leq \|\theta\|_0\leq s\right\}.\]
Finally, we define
\[\mathcal{E}\eqdef\left\{z\in\rset^n:\; \max_{1\leq k\leq d}|\pscal{X_k}{z-X\theta_\star}|\leq \lambda_1/2 \right\}.\]

\begin{theorem}\label{thm3}
Assume (\ref{dist:assump}), with a design matrix $X$ that satisfies $\underline{\kappa}>0$. Choose $\alpha=1$, $\lambda_1 = 4\sigma\sqrt{n\bar\kappa(1)\log(d)}$, and $\pi_\delta = \textsf{q}^{\|\delta\|_0}(1-\textsf{q})^{d-\|\delta\|_0}$, where $\textsf{q}=d^{-u}$, for some constant $u>1$. Suppose that $\gamma>0$ is small enough so that
\[\frac{4\gamma}{\sigma^2}\lambda_{\textsf{max}}(X'X)\leq 1,\;\;\mbox{ and }\;\;\; 24n\bar\kappa(1)\gamma\leq \sigma^2(u-1).\]
 Then if $d> 2$, $\PP(Z\in\mathcal{E})\geq 1-2/d$, and for $Z\in\mathcal{E}$,
\begin{multline}\label{thm3:rate}
\PE\left[\varrho_\gamma(Z)\vert Z\in\mathcal{E}\right] \leq -\log\left(1-\frac{2}{d}\right) + \frac{2}{d^{u-1}} + 4\log(4) + s_\star\log\left(1+\frac{\bar\kappa(s_\star)}{16\log(d)}\right) \\
+ u s_\star\log(d)
+ \left(\frac{128}{\underline{\kappa}} + \frac{96\gamma}{\sigma^2}n\right)\bar\kappa(1)\log(d) + \frac{3\gamma}{\sigma^2}\left(n\lambda_{\textsf{max}}(X'X) + \textsf{Tr}(X'X)\right). 
\end{multline}
\end{theorem}
\begin{proof}
See the Appendix in the Supplement.
\end{proof}
\begin{remark}
To give some context, the theorem considers the posterior distribution $\check\Pi(\cdot\vert Z)$, with $\alpha=1$  (Laplace prior), and $\lambda_1 = 4\sigma\sqrt{n\bar\kappa(1)\log(d)}$ in (\ref{prior:p}), and with $\pi_\delta = \textsf{q}^{\|\delta\|_0}(1-\textsf{q})^{d-\|\delta\|_0}$, where $\textsf{q}=d^{-u}$, for some constant $u>1$. It has recently been shown by  \cite{castillo:etal:14} that with the above choices, the $\theta$-marginal of this posterior distribution contracts to a point-mass at $\theta_\star$ at the optimal rate $O(\sqrt{s_\star\log(d)/n})$. Theorem \ref{thm3} gives a bound on the average error of the Moreau-Yosida approximation of this posterior distribution via
\[\PE\left[\textsf{d}_\beta\left(\check\Pi_{\gamma}(\cdot\vert Z),\check\Pi(\cdot\vert Z)\right)\vert Z\in\mathcal{E}\right] \leq \sqrt{\gamma d} +2\PE\left[\varrho_\gamma(Z)\vert Z\in\mathcal{E}\right],\] 
for $Z\in\mathcal{E}$.
\end{remark}

The right-side of (\ref{thm3:rate}) does not converge to zero as $\gamma\to 0$. However, it provides some useful insights. We see that if we choose $\gamma>0$ as in (\ref{eq:choice:gamma:2}), then the right-side of (\ref{thm3:rate}) grows with $d$ at most like $\log(d) + \textsf{Tr}(X'X)/\lambda_{\textsf{max}}(X'X)$. From random matrix theory it is known for several classes of random matrices that for fixed $n$, $\textsf{Tr}(X'X)/\lambda_{\textsf{max}}(X'X)$ is typically $O(1)$ as $d\to\infty$. An inspection of the proof of Theorem \ref{thm3} suggests that the dependence of the right-side of (\ref{thm3:rate}) on the sample size $n$ can be improved. We leave this for possible future work.

\subsection{Dealing with the hyper-parameter $\phi$}
We use a fully Bayesian for selecting the hyper-parameter $\phi=(\textsf{q},\lambda_1,\lambda_2)$. We assume independent priors such that $\textsf{q}\sim\textbf{Beta}(1,d^u)$ for some constant $u>1$, $\lambda_1\sim\textbf{U}(\textsf{a},M)$, and $\lambda_2\sim\textbf{U}(\textsf{a},M)$ for some small positive constant $\textsf{a}$ (we use $\textsf{a}=10^{-5}$ in the simulations), and for a large positive constant $M$ such that $(1-\alpha)M\leq \lambda_{\textsf{max}}(X'X)$. If $\gamma>0$ is such that (\ref{choice:gamma:2}) holds then the $\beta$-distance between the resulting posterior distribution and its Moreau-Yosida approximation satisfies the same bound as in Theorem \ref{prop1}.

\subsection{Markov Chain Monte Carlo}\label{sec:mcmc:lm}
The density $\check\pi_{\gamma}$ in (\ref{post:approx:lm}) is a ``standard" density, and various MCMC schemes can be used to sample from it. We propose a Metropolized-Gibbs strategy.
\subsubsection{Updating $\delta$}
Given $\theta$ and $\phi$, it is easy to see that $h_{\gamma}(\theta\vert\delta)$ depends on $\delta_{j}$ only through the expression
\[\delta_j\left[(\nabla \ell(\theta))_{j}d_j +\log Z(\phi) +\frac{\alpha\lambda_1|d_j| +0.5(1-\alpha)\lambda_2d_j^2}{\sigma^2} +\frac{d_j^2-2\theta_jd_j}{2\gamma}\right],\]
where $d_j$ is the $j$-th component of $\textsf{s}_\gamma(\theta-\gamma\nabla\ell(\theta);\lambda_1/\sigma^2,\lambda_2/\sigma^2)$. Hence, we update jointly and independently the $\delta_j$ by setting $\delta_{j}=1$ with probability $e^r/(1+e^r)$, where
\begin{multline*}r = \log \frac{\piq}{1-\piq} +\frac{1}{2}\log(2\pi\gamma) \\
- \left[(\nabla \ell(\theta))_{j}d_j +\log Z(\phi) +\frac{\alpha\lambda_1|d_j| +0.5(1-\alpha)\lambda_2d_j^2}{\sigma^2} +\frac{d_j^2-2\theta_jd_j}{2\gamma}\right].\end{multline*}

\subsubsection{Updating $\theta$}
Given $\delta$ and $\phi$, we update the components of $\theta$ using a mix of an independence Metropolis sampler, and a Metropolis Adjusted Langevin algorithm (MaLa). The MaLa strategy needs some motivation. Although its definition might perhaps suggest otherwise, the function $P_{\gamma}$ in (\ref{Moreau-Yosida:P}) is actually differential (\cite{bauschke:combettes:2011}~Proposition 12.29) and for all $\theta,H\in\rset^d$, 
\[\nabla_\theta P_{\gamma}(\theta\vert \delta)\cdot H=\frac{1}{\gamma}\pscal{\theta-\Prox_\gamma(\theta\vert\delta)}{H}.\]
And since  $\ell$ is twice continuously differentiable in this example, the expression (\ref{def:hng2}) shows that $h_{\gamma}$ is in fact differential and for all $\theta,H\in\rset^d$, 
\begin{eqnarray*}
\nabla_\theta h_{\gamma}(\theta\vert \delta)\cdot H &=& \pscal{\nabla\ell(\theta)}{H} -\gamma\pscal{\nabla \ell(\theta)}{\nabla^{(2)} \ell(\theta)\cdot H}\\
&&+\pscal{\frac{1}{\gamma}\left(\theta-\gamma \nabla \ell(\theta)-J_\gamma(\theta\vert\delta,\phi)\right)}{\left(I_d-\gamma\nabla^{(2)} \ell(\theta)\right)\cdot H}\\
&=&\frac{1}{\gamma}\pscal{\theta-J_\gamma(\theta\vert\delta,\phi)}{\left(I_d-\gamma\nabla^{(2)} \ell(\theta)\right)\cdot H}.\end{eqnarray*}
 To avoid dealing with second order derivatives, and since $\gamma$ is typically small, we make the approximation $I_d-\gamma\nabla^{(2)} \ell(\theta)\approx I_d$, and therefore, we approximate $\nabla_\theta h_{\gamma}(\theta\vert \delta)$ by
 \begin{equation}\label{def:G}
 G_\gamma(\theta\vert \delta)\eqdef \frac{1}{\gamma}\left(\theta-J_\gamma(\theta\vert\delta)\right),\;\mbox{ and }\;\; \bar G_\gamma(\theta\vert\delta)\eqdef \frac{\c}{\c\vee \|G_\gamma(\theta\vert \delta)\|}G_\gamma(\theta\vert \delta),
\end{equation}
for a positive constant $\c$. The function $\bar G_\gamma$ is introduced for further stability, in the spirit of the truncated Metropolis adjusted Langevin algorithm  (see e.g. \cite{atchade:2006}). Hence, given $\delta$ and $\phi$, one can update the components of $\theta$  using a Metropolized-Langevin-type algorithm where the drift function is given by the corresponding components of $\bar G_\gamma$. This algorithm is similar to the proximal MaLa of \cite{pereyra14}. 

However, when $\delta_j=0$, the corresponding component of $G_\gamma(\theta\vert \delta)$ is $\theta_j/\gamma$ and is typically very large and not very informative (particularly for $\gamma$ small). To deal with this, we use the following strategy. We update jointly the components $\theta_j$ for which  $\delta_j=1$ using the  MaLa algorithm outlined above.  Then, we group together all the components for which $\delta_j=0$ and we update them jointly using an independence Metropolis sampler. The proposal density of the Independence Metropolis sampler is built by approximating  $J_\gamma(\theta\vert\delta)$ by $\Prox_\gamma(\theta\vert\delta)$. This approximation makes sense because, for $\gamma\approx 0$, $J_\gamma(\theta\vert\delta) =\Prox_\gamma(\theta-\gamma\nabla\ell(\theta)\vert\delta)\approx \Prox_\gamma(\theta\vert\delta)$. 

To explain the detail of the independence sampler, let $\theta_\delta= (\theta_j,\;j:\;\textsf{s.t. }\delta_j=1)$, $u= (\theta_j,\;j:\; \textsf{s.t. }\delta_j=0)$, and let us represent $\theta$ by the pair $(\theta_\delta,u)$. Let $\tilde h_{\gamma}(\theta_\delta,u\vert \delta)$ be the function obtained by replacing $J_\gamma(\theta_\delta,u\vert\delta)$ by $\Prox_\gamma(\theta_\delta,u\vert\delta)$ in the expression of $h_{\gamma}(\theta_\delta,u\vert \delta)$. Because, $\Prox_\gamma(\theta_\delta,u\vert\delta)$ does not actually depend on $u$, we have
\begin{eqnarray*}
\tilde h_{\gamma}(\theta_\delta,u\vert \delta)&=& \ell(\theta) +\pscal{\nabla\ell(\theta)}{\Prox_\gamma(\theta_\delta,u\vert\delta)-\theta} +\frac{1}{2\gamma}\|\Prox_\gamma(\theta_\delta,u\vert\delta)-\theta\|^2 +\textsf{const.}\\
&=&\frac{1}{2\sigma^2}\|z-X_\delta\theta_\delta -X_{\delta^c}u\|^2 \\
&&-\frac{1}{\sigma^2}\pscal{z-X_\delta\theta_\delta -X_{\delta^c}u}{X(\Prox_\gamma(\theta_\delta,u\vert\delta)-\delta\cdot\theta) + X_{\delta^c}u} \\
&&+\frac{1}{2\gamma}\|u\|^2+\textsf{const}.
\end{eqnarray*}
It is then easy to see that $u\mapsto e^{-\tilde h_{\gamma}(\theta_\delta,u\vert \delta)}$ is proportional to the density of the Gaussian distribution
\[\textbf{N}\left(\frac{\gamma}{\sigma^2}\Sigma X_{\delta^c}'X\left(\Prox_\gamma(\theta\vert\delta)-\delta\cdot\theta\right),\gamma \Sigma\right),\;\;\mbox{ where }\;\;\Sigma\eqdef \left(I_{\|\delta^c\|}-\frac{\gamma}{\sigma^2}X_{\delta^c}'X_{\delta^c}\right)^{-1},\]
where $\delta^c$ is the vector $1-\delta$, and for any $\delta\in\Delta$, $X_\delta\in\rset^{n\times\|\delta\|}$ denote the sub-matrix of $X$ obtained by selecting the columns for which $\delta_j=1$. Notice that under the assumption $\gamma\leq \frac{\sigma^2}{4\lambda_{\textsf{max}}(X'X)}$, the matrix $\Sigma$ is always positive definite.   The acceptance probability of this independence sampler is
\[\min\left[1,\frac{\exp\left(h_{\gamma}(\theta_\delta,u'\vert\delta) - \tilde h_{\gamma}(\theta_\delta,u'\vert\delta)\right)}{\exp\left(h_{\gamma}(\theta_\delta,u\vert\delta) - \tilde h_{\gamma}(\theta_\delta,u\vert\delta)\right)}\right].\]
We found this independence sampler to be extremely efficient, with an acceptance probability typically above $90\%$. 
\subsubsection{Updating $\phi=(\textsf{q},\lambda_1,\lambda_2)$}
We update $\textsf{q}\sim \textbf{Beta}(\|\delta\|_1+1, d +d^u-\|\delta\|_1)$, and we update $(\lambda_1,\lambda_2)$ jointly using a Random Walk Metropolis algorithm with Gaussian proposal. For improved mixing, we adaptively tune the scale parameter of the proposal density.

\subsection{Simulation results and comparison with STMaLa}
We illustrate the method with a simulated data example. All the computations in this example were done using \texttt{Matlab 7.14} on a 2.8 GHz Quad-Core \texttt{Mac Pro} with 24 GB of 1066 DDR3 Ram. 

We set $n=200$, $p=500$ and we generate the design matrix $X$ by simulating the rows of $X$ independently from a Gaussian distribution with correlation $\rho^{|j-i|}$ between components $i$ and $j$. We set $\rho=0.9$.  Using $X$, we general the outcome $z=X\theta_\star +\sigma\epsilon$, with $\sigma=1$ that we assume known. We build $\theta_\star$ by randomly selecting $10$ components that we fill with draws from the uniform distribution $\epsilon\mathbf{U}(\textsf{v}/2,3\textsf{v}/2)$, where $\epsilon=\pm 1$ with probability $1/2$, all other components being set to zero.  We consider two cases for $\textsf{v}$: $\textsf{v}=1$ (that we refer to below as \textsf{SCENARIO 1}), and  $\textsf{v}=\sqrt{\log(d)/n}\approx 0.18$ (\textsf{SCENARIO 2}). \textsf{SCENARIO 2} is obviously more challenging since the average strength of the signal is at the limit of what is detectable.

We set $\gamma = \gamma_0\sigma^2/\lambda_{\textsf{max}}(X'X)$ as prescribed by (\ref{eq:choice:gamma:2}) with two choices of $\gamma_0$:  $\gamma_0=0.25$, and $\gamma_0=0.01$. 

We compare these two samplers to the STMaLa sampler of \cite{schrecketal14}. The comparison is slightly tricky because STMaLa uses a different prior, namely a Gaussian ``slab" prior. However, we expect both posterior distribution on $(\delta,\theta)$ to be close, and we expect $(\delta_\star,\theta_\star)$ to be close to the center of both distributions.
For the STMaLa, we use the \textsf{Matlab} code provided online by the authors, with the default setting. Unlike our approach, this sampler requires the true value of the sparsity parameter $\piq$, which we provide. We also edit their code to return the summary statistics presented below.

We evaluate the mixing of these samplers by computing the following two metrics along the MCMC iterations: the relative error and the $F$-score (to evaluate structure recovery), defined respectively as
\[\mathcal{E}^{(k)} = \frac{\|\theta^{(k)}-\theta_\star\|}{\|\theta_\star\|},\;\;\mbox{ and }\;\;\mathcal{F}^{(k)} = \frac{2\times \textsf{SEN}^{(k)}\textsf{PREC}{(k)}}{\textsf{SEN}^{(k)} + \textsf{PREC}{(k)}},\] where
\begin{multline}\label{sen:prec}
\textsf{SEN}^{(k)} = \frac{\sum_{j=1}^d \textbf{1}_{\{|\theta^{(k)}_{j}|>0\}}\textbf{1}_{\{|\theta_{\star,j}|>0\}}}{\sum_{j=1}^d \textbf{1}_{\{|\theta_{\star,j}|>0\}}},\; \textsf{PREC}{(k)}=\frac{\sum_{j=1}^d \textbf{1}_{\{|\theta^{(k)}_{j}|>0\}}\textbf{1}_{\{|\theta_{\star,j}|>0\}}}{\sum_{j=1}^d \textbf{1}_{\{|\theta^{(k)}_{j}|>0\}}}. \end{multline}
In stationarity we expect values of $\mathcal{E}^{(k)}$ (resp. $\mathcal{F}^{(k)}$) to be close to zero (resp. one). In the absence of a better metric, we will graphically access the mixing time of the samplers by looking at how quickly the sequence $\mathcal{E}^{(k)}$ (resp. $\mathcal{F}^{(k)}$) converges towards zero (resp. one). In order to account for the computing time, and for better comparison, we plot these metrics, not as function of the iterations $k$, but as function of the computing time needed to reach iteration $k$. For further stability in the comparison, we repeat all the samplers $30$ times and average the two metrics and the computing times over these $30$ replications.

All the chains are initialized by setting all components of $\theta^{(0)}$ (and $\delta^{(0)}$)  to zero. 
We run the samplers for a number of iterations that depends on $\theta_\star$. In \textsf{SCENARIO 1}, we run the newly proposed sampler for $10,000$, and we run STMaLa for $120,000$ iterations. In \textsf{SCENARIO 2},  we run our proposed sampler for $40,000$, and we run STMaLa for $250,000$ iterations.

Figure \ref{Fig:1} and \ref{Fig:2} present the results. First, we observe that that $\gamma_0=0.25$ mixes significantly  better than $\gamma_0=0.01$. We notice also that $\check\Pi_{\gamma}$ approximates $(\theta_\star,\delta_\star)$ only slightly better when $\gamma_0=0.01$ compared to $\gamma_0=0.25$. Overall, we found that $\gamma_0\in (0.1,0.25)$ produces a very good approximation. 

\begin{figure}
\scalebox{0.82}{\includegraphics[width = 7in]{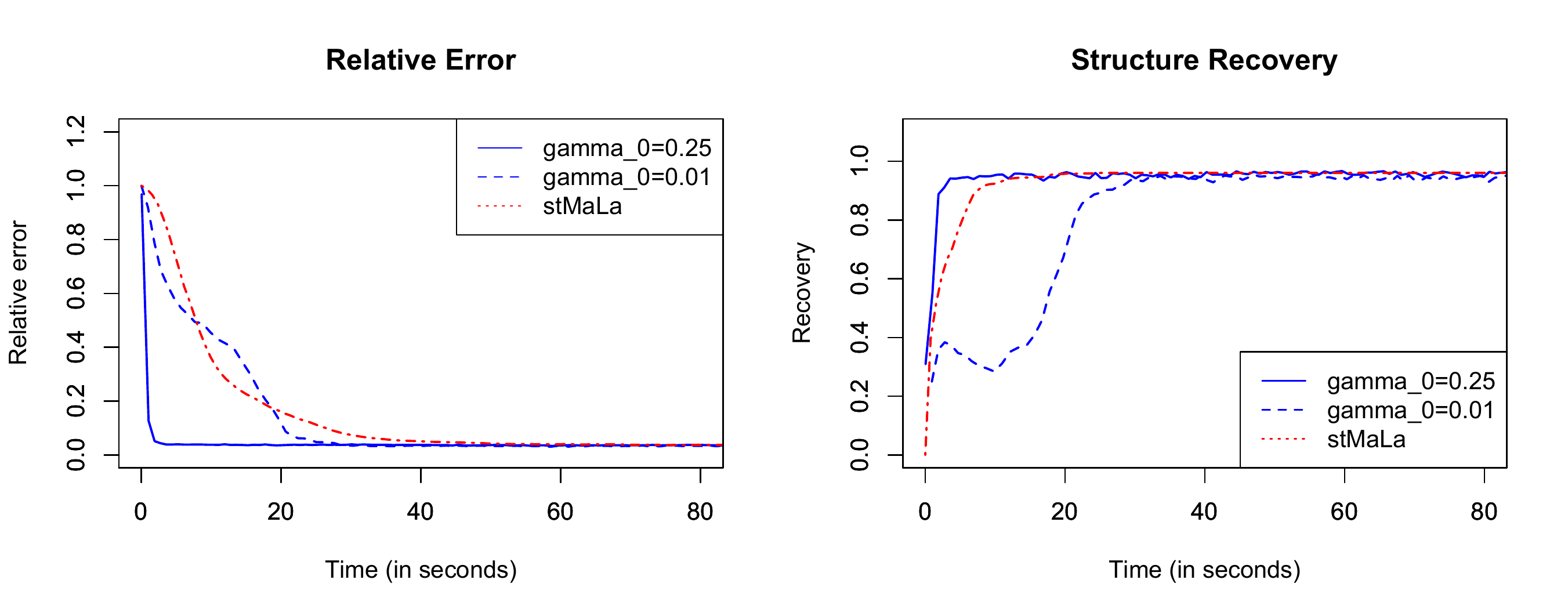}}
\caption{Relative error and structure recovery as function of time in \textsf{SCENARIO 1}. Based on 30 MCMC replications. The curves are sub-sampled to improve the readability of the figure.}
\label{Fig:1}
\end{figure}

\begin{figure}
\scalebox{0.82}{\includegraphics[width = 7in]{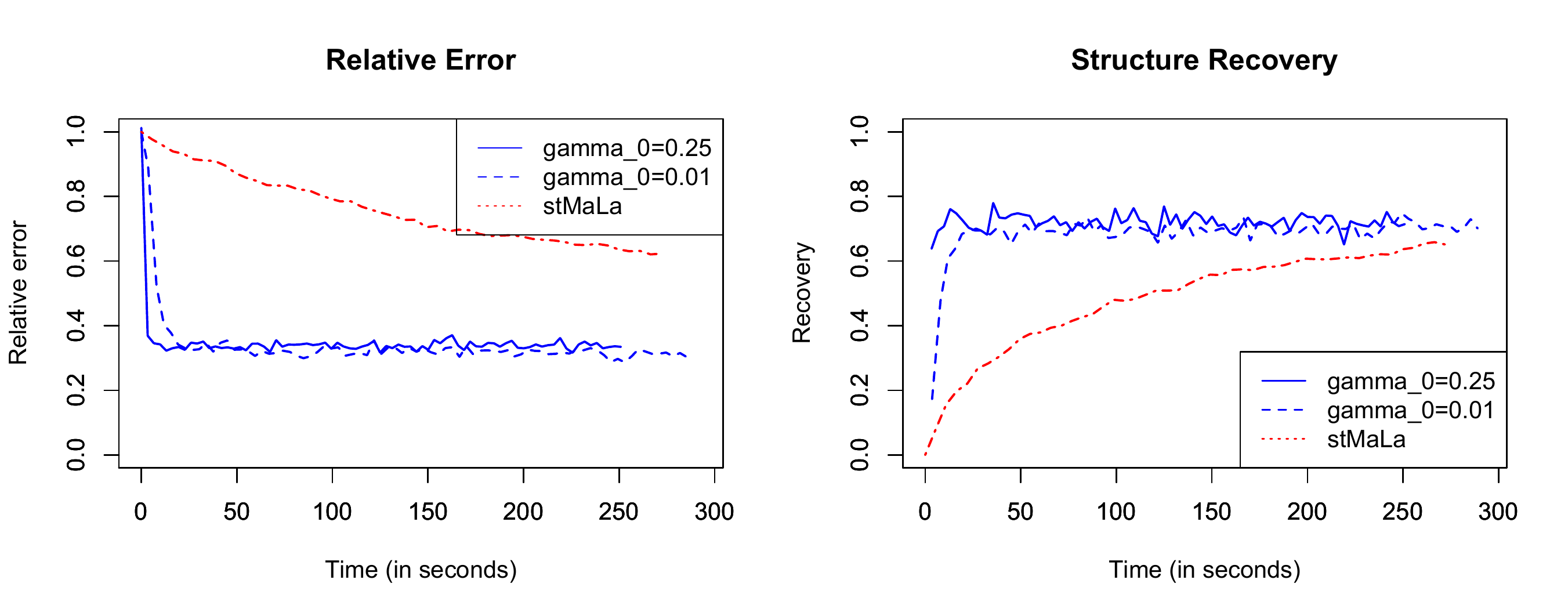}}
\caption{Relative error and structure recovery as function of time in \textsf{SCENARIO 2}. Based on 30 MCMC replications. The curves are sub-sampled to improve the readability of the figure.}
\label{Fig:2}
\end{figure}

We also look at the usual sample path mixing of the proposed sampler by plotting the trace plot, histogram, and the autocorrelation plot from  a single run of the sampler (Figure \ref{Fig:4}). Here, we consider only  \textsf{SCENARIO 1}, and we set $\gamma_0=0.25$. We look at the MCMC output $\{\theta^{(k)}_j,\;k\geq 0\}$, for one component $j$ for which $\delta_j=0$, and for one component $j$ for which $\delta_j=1$. From this sample path perspective, the plots suggest that the proposed MCMC sampler has a good mixing.


\begin{figure}
\scalebox{0.8}{\includegraphics[width = 7in]{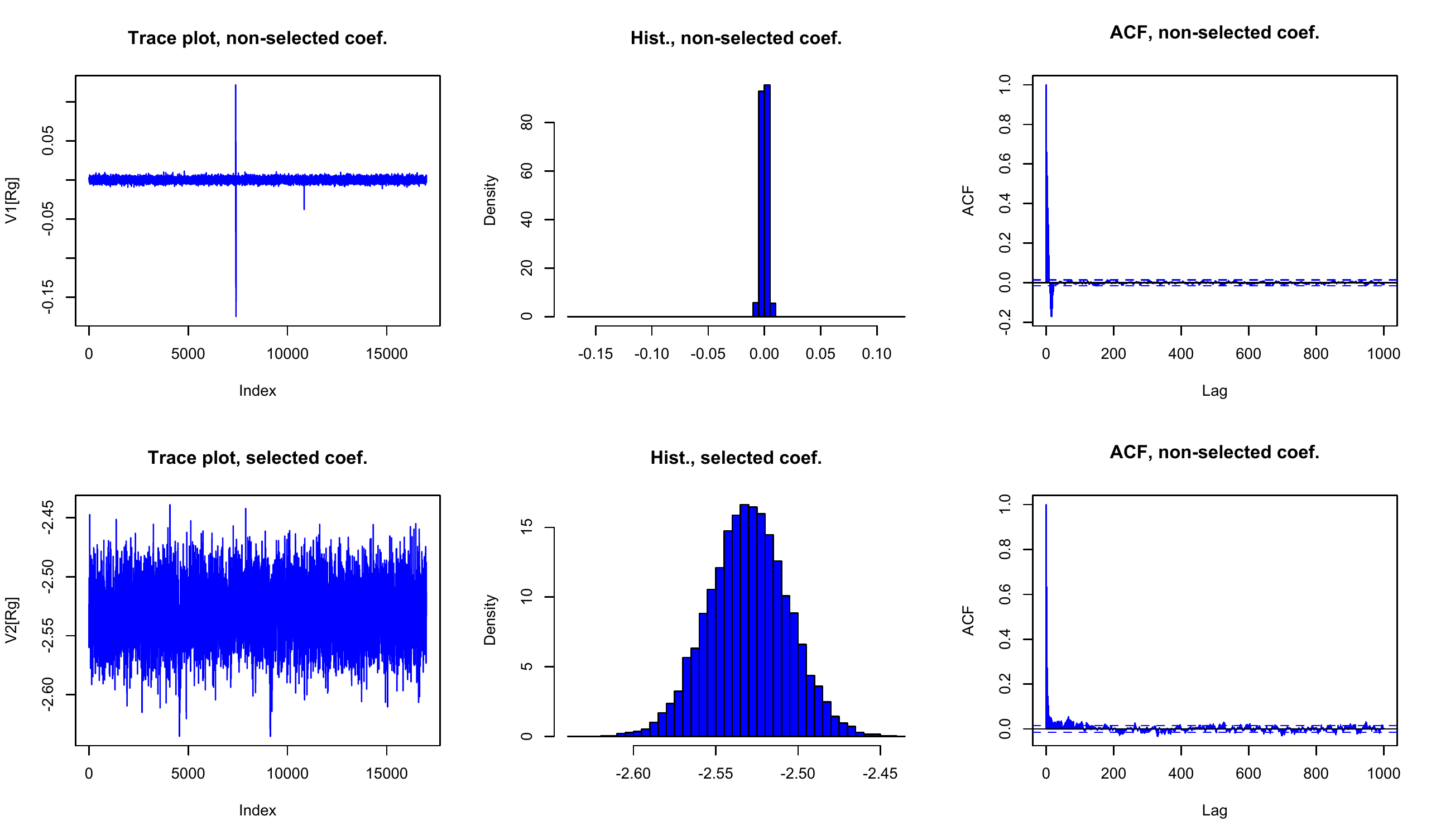}}
\caption{Trace plot, histogram, and autocorrelation plot, from one MCMC run, using $\gamma_0=0.25$. Top row is for a component $j$ for which the true value of $\delta_j$ is $0$. Bottom row, true value of $\delta_j$ is $1$.}
\label{Fig:4}
\end{figure}

\subsection{Empirical Bayes implementation and further experimentation}\label{sec:EB}
A limitation of the methodology is that $\sigma^2$ is assumed known, which is rarely the case in practice. We explore by simulation an empirical Bayes solution whereby $\sigma^2$ is estimated from data. Following \cite{reid:etal:13} we estimate $\sigma^2$ by
\[\hat\sigma^2_n = \frac{1}{n-\hat s_{\lambda_n}}\sum_{i=1}^n\left(y_i - x_i\hat\beta_{\lambda_n}\right)^2,\]
where $\hat\beta_{\lambda}$ is the lasso estimate at regularization level $\lambda$, and $\lambda_n$ is selected by 10-fold cross-validation, and where $\hat s_{\lambda_n}$ is the number of non-zeros components of $\hat\beta_{\lambda_n}$. In the cross-validation, we choose $\lambda_n$ as the value of $\lambda$ that minimizes the MSE. This leads to the empirical Bayes Moreau-Yosida posterior approximation that we denote $\check\Pi_{\gamma}(\cdot\vert z,\hat\sigma_n^2)$. We do a simulation study using a semi-real dataset to compare the distributions $\check\Pi_{\gamma}(\cdot\vert z,\hat\sigma_n^2)$ and $\check\Pi_{\gamma}(\cdot\vert z)$ (with the true value of $\sigma^2$ set to one). We use the \textsf{colon dataset} (\cite{buhlmann:mandozzi:14}) downloaded from\\ 
\verb|http://stat.ethz.ch/~dettling/bagboost.html|. The data gives microarray gene expression levels for $2,000$ genes for $n=62$ patients in a colon cancer study. We randomly select a subset of $p=1,000$ variables to form a design matrix $X\in\rset^{62\times 1,000}$. Following \cite{buhlmann:mandozzi:14}, we normalize each column of $X$ to have mean zero and variance unity. We simulate a sparse signal vector $\theta_\star\in\rset^p$ with $s=5$ non-zeros components, and where the non-zeros components are drawn from $\textbf{U}(-\textsf{v}-1,-\textsf{v})\cup (\textsf{v},\textsf{v}+1)$. We consider two scenarios: $\textsf{v}=1$ and $\textsf{v}=3$. Using $X$ and $\theta_\star$, we generate $z=X\theta_\star +\sigma\epsilon$, with $\sigma=1$, and $\epsilon\sim\textbf{N}(0,I_n)$.

We set $\gamma$ as in (\ref{eq:choice:gamma:2}) with $\gamma_0=0.25$. We evaluate the samplers along the same metrics $\mathcal{E}$ and $\mathcal{F}$. We average the results over 30 replications\footnote{here only $X$ and $\theta_\star$ are kept fixed. For each replication, the dataset $z$ is re-simulated, and $\sigma^2_n$ is re-estimated.} of the samplers, where each sampler is run for $50,000$ iterations. The result is presented on Table \ref{table:1}. We notice that the recoery of $\theta_\star$ is poor in both cases when $\textsf{v}=1$. When the signal is strong ($\textsf{v}=3$), the empirical Bayes posterior distribution performs well, but as expected, under-performs the posterior distribution with known variance. 



\begin{table}[h]
\begin{center}
\small
\scalebox{.9}{\begin{tabular}{ccccc} \hline
 & \multicolumn{2}{c}{Weak signal ($\textsf{v}=1$)}& \multicolumn{2}{c}{Strong signal ($\textsf{v}=3$)} \\
    &    EB   & True $\sigma$    &   EB     & True $\sigma$ \\ \hline
Relative error (in $\%$) & 97.3 & 91.7 & 12.4  & 9.4 \\
$F$-score ( in $\%$) &  14.5  & 25.1  & 79.6 &   88.5\\
\hline
\end{tabular}}
\caption{\small{Table showing the posterior estimates $(N-B)^{-1}\sum_{k=B+1}^N\mathcal{E}^{(k)}$, and $(N-B)^{-1}\sum_{k=B+1}^N\mathcal{F}^{(k)}$, averaged over 30 MCMC replications, each MCMC run is $5\times 10^4$ iterations.
}}\label{table:1}
\end{center}
\end{table}

\section{Further Discussion}\label{sec:comments}
In this work we have developed and analyzed  a smooth approximation to high-dimensional posterior distribution using the Moreau-Yosida envelop. The methodology can be readily extended to other high-dimensional statistical models (linear and generalized linear regression models, graphical models, sparse PCA, and others). Several theoretical issues remain. We have discussed some of these issues above. One important problem that we did not directly address concerns the mixing properties of the proposed MCMC algorithms, and the trade-off inherent to the methodology between good approximation properties of $\check\Pi_\gamma$, and good mixing of gradient-based MCMC simulation from $\check\Pi_\gamma$. Another potentially interesting direction of research is the idea of treating $\check\Pi_\gamma$ itself as a quasi-posterior distribution, and investigating directly its posterior contraction properties.

\section{APPENDIX: Proof of the main results}\label{sec:proofs}
For convenience, we introduce the product space $\bar\Theta\eqdef \Delta\times\rset^d$ that we implicitly equip with the metric $\dist_{\bar\Theta}(\bar\theta_1,\bar\theta_2)\eqdef \sqrt{\|\delta_1-\delta_2\|_0^2 + \|\theta_1-\theta_2\|^2}$, $\bar\theta_j=(\delta_j,\theta_j)$, $j=1,2$.

\subsection{Proof of Proposition \ref{prop0}}\label{proof:prop0}
For all $x\in\rset^d$, and $\gamma\in(0,\gamma_0]$,  $e^{-h(x)}\leq e^{-h_\gamma(x)}\leq e^{-h_{\gamma_0}(x)}$. Hence $Z\leq Z_\gamma\leq Z_{\gamma_0}<\infty$. Since $\mu$ is the Lebesgue measure on $\rset^d$, we shall write it as $\rmd x$.  For any bounded measurable function $f:\;\rset^d\to\rset$, we have
\begin{eqnarray*}
\left|\Pi_\gamma(f) - \Pi(f)\right| &\leq &\frac{1}{Z_\gamma}\left|\int_{\rset^d} f(x)\left(e^{-h_\gamma(x)}-e^{-h(x)}\right)\rmd x\right|\\
 && +\frac{(Z_\gamma -Z)}{Z_\gamma Z}\int_{\rset^d} |f(x)|e^{-h(x)}\rmd x,\\
 & \leq  &  \frac{2\|f\|_\infty}{Z_\gamma} \int_{\rset^d}\left(e^{-h_\gamma(x)} - e^{-h(x)}\right)\rmd x\\
 & = & 2\|f\|_\infty \left(1-\frac{Z}{Z_\gamma}\right).
\end{eqnarray*}
The fact that $Z_\gamma \to Z$ as $\gamma\downarrow 0$, follows from Lebesgue's monotone convergence applied to $e^{-h_{\gamma_0}} -e^{-h_\gamma}$.

\subsection{Proof of Theorem \ref{thm1}}\label{proof:thm1}
We work on the product space $\bar\Theta =\Delta\times \rset^d$ introduced above. Throughout the proof, we assume that $z$ is fixed, and at times we write $\check\Pi(\cdot\vert z)$ simply as $\check\Pi$. Same for $\tilde\Pi_\gamma(\cdot\vert z)$ and $\check\Pi_\gamma(\cdot\vert z)$. 

We prove the theorem in two steps. First in Lemma \ref{lem1}, we bound the Wasserstein distance between the distributions $\tilde\Pi_\gamma$ and $\check\Pi$ by showing that for all $\gamma>0$,
\begin{equation}\label{eq:step1}
\textsf{d}_{\textsf{w}}(\tilde\Pi_\gamma,\check\Pi)\leq \sqrt{\gamma d}.
\end{equation}
Then  in Lemma \ref{lem3:thm1} we bound the total variation distance between $\check\Pi_\gamma$ and $\tilde\Pi_\gamma$ by showing that for all $\gamma\in(0,\gamma_0]$,
\begin{equation}\label{eq:step2}
\textsf{d}_{\textsf{tv}}(\tilde\Pi_\gamma,\check\Pi_\gamma)\leq 2\left(1-e^{-\varrho_\gamma(z)}\right).
\end{equation}
It is clear from their definitions that both the Wasserstein metric and the total variation metric are upper bounds for the metrix $\beta$, and the Theorem \ref{thm1} follows by combining (\ref{eq:step1}) and (\ref{eq:step2}). The proof of Lemma \ref{lem3:thm1} relies on a comparison result between the functions $h$ and $h_{\gamma}$ established in Lemma \ref{lem2:thm1} that is also of independent interest.

\begin{lemma}\label{lem1}
Let $\tilde\Pi_\gamma$ be the probability measure defined in (\ref{tilde:pi}). For all $\gamma>0$,
\begin{equation}
\sqrt{\frac{2}{\pi}}\sqrt{\gamma d}\left(1-\frac{1}{d}\PE(\|\eta\|_0)\right)\leq \textsf{d}_{\textsf{w}}(\tilde\Pi_\gamma,\check\Pi)\leq \sqrt{d\gamma}\sqrt{1-\frac{1}{d}\PE(\|\eta\|_0)},
\end{equation}
where $\eta$ is a random variable on $\Delta$ with distribution given by the $\delta$-marginal of $\check\Pi$, that is $\PP(\eta=\delta)\propto \pi_\delta\int_{\rset^d}e^{-h(\theta\vert\delta)}\mu_\delta(\rmd\theta)$, $\delta\in\Delta$.
\end{lemma}
\begin{proof}
For $\delta\in\Delta$, we set 
\begin{equation}\label{def:Cn}
 C(\delta) \eqdef \int_{\rset^d} e^{-h(\theta\vert \delta)} \mu_{\delta}(\rmd\theta),\;\;\mbox{ and } C=\sum_{\delta\in\Delta}\pi_\delta C(\delta).\end{equation}
Using this notation, we can write
\[\check\Pi(\delta,\rmd\theta\vert z)=\frac{\pi_\delta C(\delta)}{C}\check\Pi(\rmd \theta\vert \delta,z),\;\mbox{ where }\;\; \check\Pi(\rmd \theta\vert \delta,z) \eqdef \frac{1}{C(\delta)}e^{-h(\theta\vert \delta)} \mu_{\delta}(\rmd\theta).\]
For $\gamma>0$, we notice that the normalizing constant of $\tilde\Pi_\gamma$ is 
\begin{eqnarray}\label{def:Cn0}
C &\eqdef & \sum_\delta\pi_\delta \left(\frac{1}{2\pi\gamma}\right)^{\frac{d-\|\delta\|_1}{2}} \int_{\rset^d}  e^{-\frac{1}{2\gamma}\|\theta-\theta_\delta\|^2} e^{-h(\theta_\delta\vert\delta)}\rmd\theta \nonumber\\
& = & \sum_\delta\pi_\delta \int_{\rset^d} e^{-h(\theta\vert\delta)}\mu_{\delta}(\rmd\theta),
\end{eqnarray}
which is the same as the normalizing constant of the posterior distribution $\check\Pi$. Hence we get the following factorization of $\tilde\Pi_\gamma$,
\begin{multline*}
\tilde\Pi_\gamma(\delta,\rmd\theta\vert z)= \frac{\pi_\delta C(\delta)}{C} \tilde\Pi(\rmd \theta\vert \delta,z),\;\;\mbox{ where }\;\; \tilde\Pi(\rmd \theta\vert \delta,z) \eqdef \frac{1}{C(\delta)}e^{-\frac{1}{2\gamma}\|\theta-\theta_\delta\|^2} e^{-h(\theta_\delta\vert \delta)} \rmd\theta.\end{multline*}
We build following coupling of $\check\Pi$ and $\tilde\Pi_\gamma$. First we generate $\eta\in\Delta$ from the  distribution $\delta \mapsto \frac{\pi_\delta C(\delta)}{C}$, and we generate $\check\vartheta\vert\eta\sim  \check\Pi(\rmd\theta\vert z,\eta)$. Hence clearly, $(\eta,\check\vartheta)\sim\check\Pi$. Given $(\eta,\check\vartheta)$, we generate $\tilde\vartheta$ as follows. If $\eta_j=1$, we set $\tilde\vartheta_j=\check\vartheta_j$. Otherwise we generate independently $Z_j\sim\textbf{N}(0,1)$, and set $\tilde\vartheta_j=\sqrt{\gamma}Z_j$. It is also easy to check that $(\eta,\tilde\vartheta)\sim\tilde\Pi_\gamma$.

For any Lipschitz function on $\bar\Theta$ with Lipschitz constant less of equal to $1$, we have
\begin{multline*}
\left|\int f(\delta,\theta)\tilde\Pi_\gamma(\rmd\delta,\rmd\theta) -\int f(\delta,\theta)\check\Pi(\rmd\delta,\rmd\theta) \right| = \left|\PE\left[f(\eta,\tilde\vartheta) - f(\eta,\check\vartheta)\right]\right|\\
\leq \PE\left[\|\tilde\vartheta - \check\vartheta\|\right]\leq \sqrt{d\gamma}\sqrt{1-\frac{1}{d}\PE(\|\eta\|_0)}.
\end{multline*}
Now consider the function $f_0(\delta,\theta) = \frac{1}{\sqrt{d}}\sum_{j=1}^d |\theta_j|$. It is Lipschitz with Lipschitz constant $1$. Hence
\begin{multline*}
\textsf{d}_{\textsf{w}}(\tilde\Pi_\gamma,\check\Pi)\geq \left|\PE\left[f_0(\eta,\tilde\vartheta) - f_0(\eta,\check\vartheta)\right]\right| =\sqrt{\frac{\gamma}{d}}\PE\left(\sum_{j:\;\eta_j=0} |Z_j|\right)\\
=\sqrt{\frac{2}{\pi}}\sqrt{\gamma d}\left(1-\frac{1}{d}\PE(\|\eta\|_0)\right),
\end{multline*}
and the result is proved.
\end{proof}

\begin{lemma}\label{lem2:thm1}
Assume H\ref{A1} and fix $\delta\in\Delta$. For all $\theta\in\rset^d$,
\begin{equation}\label{eq:lem2:thm1}
h(\theta_\delta\vert\delta)+\frac{1}{2\gamma}\|\theta-\theta_\delta\|^2\geq h_{\gamma}(\theta\vert\delta) \geq  h(\theta_\delta\vert\delta) +\frac{1}{2\gamma}\|\theta-\theta_\delta\|^2 -r_\gamma(\theta\vert\delta),
\end{equation}
with 
\[r_\gamma(\delta,\theta) \eqdef \pscal{\nabla\ell(\theta)-\nabla\ell(\theta_\delta)}{\theta-\theta_\delta)} + \frac{\gamma}{2}\|\delta\cdot \nabla\ell(\theta)+\delta\cdot g\left(\theta_ \delta\vert\delta\right)\|^2,\] 
and where $g(\theta_\delta\vert\delta)$ denotes a sub-gradient of $P(\cdot\vert\delta)$ at $\theta_\delta$. It follows in particular that for all $\theta\in\rset^d$, $h_{\gamma}(\theta\vert\delta)\uparrow h(\theta\vert\delta)$, as $\gamma\downarrow 0$.
\end{lemma}
\begin{proof}
From the definition we have
\begin{eqnarray*}
h_{\gamma}(\theta\vert\delta)&=&\min_{u\in\rset^d}\left[ \ell(\theta) +\pscal{\nabla \ell(\theta)}{u-\theta} +P(u\vert\delta) +\frac{1}{2\gamma}\|u-\theta\|^2\right]\\
&\leq & \ell(\theta) + \pscal{\nabla\ell(\theta)}{\theta_\delta-\theta} +P(\theta_\delta\vert\delta) + \frac{1}{2\gamma}\|\theta-\theta_\delta\|^2.
\end{eqnarray*}
By convexity of $\ell$, $\ell(\theta)+ \pscal{\nabla\ell(\theta)}{\theta_\delta-\theta}\leq \ell(\theta_\delta)$, which proves the first inequality in (\ref{eq:lem2:thm1}). To prove the second inequality, we start by using again the convexity of $\ell$ to write  for all $\theta\in\rset^d$,
\[\ell(\theta) \geq \ell(\theta_\delta) +\pscal{\nabla\ell(\theta_\delta)}{\theta-\theta_\delta}.\]
Hence for all $\theta\in\rset^d$,
\begin{multline}\label{eq1:lem2:thm1}
\ell(\theta) +\pscal{\nabla\ell(\theta)}{J_\gamma(\theta\vert\delta)-\theta} \geq \ell(\theta_\delta)+\pscal{\nabla\ell(\theta_\delta)-\nabla\ell(\theta)}{\theta-\theta_\delta}\\
+\pscal{\nabla\ell(\theta)}{J_\gamma(\theta\vert\delta)-\theta_\delta}. \end{multline}
By H\ref{A1}, $P(\cdot\vert\delta)$ is convex, and if $g(\theta_\delta\vert\delta)$ denotes a sub-gradient of $P(\cdot\vert\delta)$ at $\theta_\delta$, we have
\begin{equation}\label{eq2:lem2:thm1}
P(J_\gamma(\theta\vert\delta)\vert\delta) \geq P(\theta_\delta\vert\delta) +\pscal{g\left(\theta_ \delta\vert\delta\right)}{J_\gamma(\theta\vert\delta)-\theta_\delta}.\end{equation}
 (\ref{eq1:lem2:thm1})-(\ref{eq2:lem2:thm1}) together with the expression (\ref{def:hng3}) of $h_{\gamma}$ imply that
\begin{multline*}
h_{\gamma}(\theta\vert\delta) \geq h(\theta_\delta\vert\delta) -\pscal{\nabla\ell(\theta)-\nabla\ell(\theta_\delta)}{\theta-\theta_\delta} \\
+ \pscal{\nabla\ell(\theta)+g\left(\theta_ \delta\vert\delta\right)}{J_\gamma(\theta\vert\delta)-\theta_\delta}  +\frac{1}{2\gamma}\|\theta-J_\gamma(\theta\vert\delta)\|^2.
\end{multline*}
Since $J_\gamma(\theta\vert\delta)\in\rset^d_\delta$, we can split $\|\theta-J_\gamma(\theta\vert\delta)\|^2$ as $\|\theta-\theta_\delta\|^2 + \|\theta_\delta - J_\gamma(\theta\vert\delta)\|^2$. We use this in the last inequality to conclude that 
\begin{eqnarray*}
h_{\gamma}(\theta\vert\delta) &\geq& h(\theta_\delta\vert\delta) +\frac{1}{2\gamma}\|\theta-\theta_\delta\|^2-\pscal{\nabla\ell(\theta)-\nabla\ell(\theta_\delta)}{\theta-\theta_\delta}\nonumber \\
&&+ \pscal{\nabla\ell(\theta)+g\left(\theta_\delta\vert\delta\right)}{J_\gamma(\theta\vert\delta)-\theta_\delta} +\frac{1}{2\gamma}\|J_\gamma(\theta\vert\delta)-\theta_\delta\|^2\nonumber\\
&\geq & h(\theta_\delta\vert\delta) +\frac{1}{2\gamma}\|\theta-\theta_\delta\|^2-\pscal{\nabla\ell(\theta)-\nabla\ell(\theta_\delta)}{\theta-\theta_\delta}\nonumber \\
&&-\frac{\gamma}{2}\|\delta\cdot\nabla\ell(\theta)+\delta\cdot g\left(\theta_ \delta\vert\delta\right)\|^2,
\end{eqnarray*}
as claimed. In the last inequality, the $\delta$ appearing in front of $\nabla\ell(\theta) + g(\theta\vert\delta)$ comes from the fact that $J_\gamma(\theta\vert\delta)-\theta_\delta\in\rset^d_\delta$.

It is obvious from its definition that $h_{\gamma}(\theta\vert\delta)$ is non-decreasing as $\gamma\downarrow 0$. If $\theta\notin\rset^d_\delta$, then $\|\theta-\theta\cdot\delta\|>0$, and then both extreme sides of (\ref{eq:lem2:thm1}) converges to $+\infty=h(\theta\vert\delta)$ as $\gamma\downarrow 0$.  If $\theta\in\rset^d_\delta$, then $\|\theta-\theta\cdot\delta\|=0$ and both extreme sides of (\ref{eq:lem2:thm1}) converges to $h(\theta\cdot\delta\vert\delta)=h(\theta\vert\delta)$ as $\gamma\downarrow 0$. 
\end{proof}

\begin{lemma}\label{lem3:thm1}
Assume H\ref{A1}. Suppose that there exists $\gamma_0>0$ such that $\check\Pi_{\gamma_0}(\cdot\vert z)$ is well-defined. Then for all $\gamma\in (0,\gamma_0]$, $\check\Pi_{\gamma}(\cdot\vert z)$ is well-defined  and
\begin{equation}\label{bound:lem4:thm1}
\dist_{\tv}(\check\Pi_{\gamma},\tilde\Pi_{\gamma})\leq  2\left(1-e^{-\varrho_\gamma(z)}\right).\end{equation}
\end{lemma}
\begin{proof}
For all $\gamma>0$, we define
\[C_{\gamma}(\delta) \eqdef \int_{\rset^d} e^{-h_{\gamma}(\theta\vert \delta)} \rmd\theta,\;\;\mbox{ and }\;\; C_{\gamma} = \sum_\delta \pi_\delta (2\pi\gamma)^{\frac{\|\delta\|_0}{2}}C_{\gamma}(\delta).\]
The term $C_{\gamma}$ is the normalizing constant of $\check\Pi_{\gamma}$. The function $h_{\gamma}$ is nondecreasing as $\gamma\downarrow 0$. Hence, if $C_{\gamma_0}<\infty$, then $C_\gamma<\infty$ for all $\gamma\in (0,\gamma_0]$, which guarantees that $\check\Pi_\gamma$ is well-defined for all $\gamma\in (0,\gamma_0]$. For the remaining of the proof, we fix $\gamma\in(0,\gamma_0]$.
 To derive the total variation majoration, we start with a bound on $C_\gamma$.  Using the second inequality of (\ref{eq:lem2:thm1}), we  write 
\begin{eqnarray*}
(2\pi\gamma)^{-d/2}C_\gamma &=&\sum_\delta \pi_\delta \left(\frac{1}{2\pi\gamma}\right)^{\frac{d-\|\delta\|_0}{2}} \int_{\rset^d} e^{-h_\gamma(\theta\vert \delta)} \rmd\theta\\
&\leq & \sum_\delta \pi_\delta \left(\frac{1}{2\pi\gamma}\right)^{\frac{d-\|\delta\|_0}{2}} \int_{\rset^d} e^{r_\gamma(\delta,\theta)} e^{-\frac{1}{2\gamma}\|\theta-\theta_\delta\|^2} e^{-h(\theta_\delta\vert \delta)} \rmd\theta.
\end{eqnarray*}
where 
\[r_\gamma(\delta,\theta) = \pscal{\nabla\ell(\theta)-\nabla\ell(\theta_\delta)}{\theta-\theta_\delta)} + \frac{\gamma}{2}\|\delta\cdot \nabla\ell(\theta) +\delta\cdot g\left(\theta_ \delta\vert\delta\right)\|^2.\]
In view of this last inequality, and the definitions of $\tilde\Pi_\gamma$, and $\varrho_\gamma$, we get
\begin{equation}\label{control:C}
\frac{(2\pi\gamma)^{-d/2}C_\gamma}{C} \leq  e^{\varrho_\gamma(z)}.
\end{equation}
The total variation bound between $\tilde\Pi_{\gamma}(\delta,\rmd\theta\vert z)$ and $\check \Pi_{\gamma}(\delta,\rmd\theta\vert z)$ now follows from a comparison of the two measures. Indeed,
 Using the first inequality of (\ref{eq:lem2:thm1}), and for $\gamma\in(0,\gamma_0]$, we deduce that
 \begin{eqnarray}\label{mino}
 \check\Pi_{\gamma}(\delta,\rmd\theta\vert z) & = & \frac{1}{C_{\gamma}}\pi_\delta \left(\frac{1}{2\pi\gamma}\right)^{-\frac{\|\delta\|_1}{2}} e^{-h_{\gamma}(\theta\vert\delta)}\rmd\theta \theta\nonumber\\
 & \geq & \frac{1}{C_{\gamma}}\pi_\delta \left(\frac{1}{2\pi\gamma}\right)^{-\frac{\|\delta\|_1}{2}} e^{-\frac{1}{2\gamma}\|\theta-\theta_\delta\|^2} e^{-h(\theta_\delta\vert\delta)}\rmd\theta \nonumber\\
 & = & \frac{C}{\left(2\pi\gamma\right)^{-\frac{d}{2}}C_{\gamma}} \tilde\Pi_{\gamma}(\delta,\rmd\theta\vert z),\nonumber\\
 & \geq & e^{-\varrho_\gamma(z)} \tilde\Pi_{\gamma}(\delta,\rmd\theta\vert z),
 \end{eqnarray}
using (\ref{control:C}). By a standard coupling argument (see e.g. \cite{lindvall92}~Equation 5.1), the minorization (\ref{mino}) implies (\ref{bound:lem4:thm1}).
\end{proof}

\subsection{Proof of Theorem \ref{thm2}}\label{proof:thm2}
It suffices to establish the stated bound on $\varrho_\gamma(z)$, and apply Theorem \ref{thm1}.
From its definition, we have
\[e^{\varrho_\gamma(z)} = \frac{\sum_{\delta\in\Delta} \pi_\delta \left(\frac{1}{2\pi\gamma}\right)^{\frac{d-\|\delta\|_0}{2}}\int_{\rset^d} e^{r_\gamma(\delta,\theta)} e^{-\frac{1}{2\gamma}\|\theta-\theta_\delta\|^2} e^{-h(\theta_\delta\vert \delta)} \rmd\theta}{\sum_{\delta\in\Delta} \left(\frac{1}{2\pi\gamma}\right)^{\frac{d-\|\delta\|_0}{2}}\int_{\rset^d} e^{-\frac{1}{2\gamma}\|\theta-\theta_\delta\|^2} e^{-h(\theta_\delta\vert \delta)} \rmd\theta}.\] 
It follows from H\ref{A2} that 
\begin{eqnarray}\label{eq:control:r}
r_\gamma(\delta,\theta) &\leq& \pscal{\nabla\ell(\theta)-\nabla\ell(\theta_\delta)}{\theta-\theta_\delta} +\frac{3 \gamma}{2} \|\nabla\ell(\theta)-\nabla\ell(\theta_\delta)\|^2  \nonumber\\
&&+ \frac{3 \gamma}{2} \|\delta\cdot\nabla\ell(\theta_\delta)\|^2 + \frac{3 \gamma}{2}\|g(\theta_\delta\vert\delta,\phi)\|^2\nonumber\\
&\leq&  L_1\left (1+\frac{3\gamma}{2}L_1\right)\|\theta-\theta_\delta\|^2  +3\gamma L_2 \ell(\theta_\delta) +\frac{3\gamma}{2}c(\delta)+3\gamma L_2 P(\theta_\delta\vert\delta).\end{eqnarray}
We set $h_\gamma \eqdef 1-2\gamma L_1\left (1+\frac{3\gamma}{2}L_1\right)$, and $a\eqdef 3L_2$. Then (\ref{eq:control:r}) gives
\begin{multline}\label{eq:control:R1}
\int_{\rset^d} e^{r_\delta(\delta,\theta)} e^{-\frac{1}{2\gamma}\|\theta-\theta_\delta\|^2} e^{h(\theta_\delta\vert \delta)}\rmd\theta \leq e^{\frac{3\gamma}{2}c(\delta)}\\
 \times\int_{\rset^d}e^{-\frac{h_\gamma}{2\gamma}\|\theta-\theta_\delta\|^2} e^{-(1-\gamma a)\ell(\theta_\delta) -(1-\gamma a)P(\theta_\delta\vert \delta)} \rmd\theta.\end{multline}

Notice that the integral on the right-side of (\ref{eq:control:R1}) can be factorized as the product of two integrals, with one integral taken over the components for which $\delta_j=0$, and the other taken over the components for which $\delta_j=1$. We introduce some notation to do this rigorously. Fix $\delta\in\Delta$, and $s=\|\delta\|_0$. For a given function $f:\;\rset^d\to\rset$, we define $f^{[s]}:\;\rset^{s}\to\rset$ as $f^{[s]}(u) =f(u^\delta)$, where $u^\delta\in\rset^d$, and $u^\delta_i=0$ if $\delta_i=0$, and $u_j^\delta = u_{\sum_{k=1}^j \delta_k}$ if $\delta_j=1$.  With this notation, and for $4\gamma L_1\leq 1$ (which implies that $h_\gamma>0$), the integral on the right-hand side of (\ref{eq:control:R1}) is equal to
\[\left(\frac{2\pi\gamma}{h_\gamma}\right)^{\frac{d-s}{2}}\int_{\rset^s} e^{-(1-\gamma a)\ell^{[s]}(u) -(1-\gamma a)P^{[s]}(u\vert \delta)}\rmd u.\]
 A similar calculation on the denominator of $e^{\varrho_\gamma(z)}$ gives
\[ \int_{\rset^d} e^{-\frac{1}{2\gamma}\|\theta-\theta_\delta\|^2} e^{-h(\theta_\delta\vert \delta)}\rmd\theta = (2\pi\gamma)^{\frac{d-s}{2}} \int_{\rset^s} e^{-\ell^{[s]}(u) -P^{[s]}(u\vert\delta)}\rmd u.\]
We conclude that
\begin{equation}\label{eq:control:R2}
e^{\varrho_\gamma(z)} \leq \frac{\sum_{\delta\in\Delta}\pi_\delta e^{\frac{3\gamma}{2} c(\delta)} \left(\frac{1}{h_\gamma}\right)^{\frac{d-s}{2}}\int_{\rset^s} e^{-(1-\gamma a)\ell^{[s]}(u) -(1-\gamma a)P^{[s]}(u\vert \delta)}\rmd u}{\sum_{\delta\in\Delta} \pi_\delta \int_{\rset^s} e^{-\ell^{[s]}(u) -P^{[s]}(u\vert\delta)}\rmd u},\end{equation}

For $4\gamma L_1\leq 1$, and using the inequality $\log(1-2x-3x^2)\geq -6x$, valid for all $x\in [0,1/4]$, we have
\begin{equation}\label{control:h_gamma}
\left(\frac{1}{h_\gamma}\right)^{\frac{d-s}{2}} = \exp\left[-\frac{d-s}{2}\log\left(1-2\gamma L_1 - 3\gamma^2L_1^2\right)\right] \leq e^{3d\gamma L_1}.\end{equation}

Fix $u_0\in\rset^s$, arbitrary. Since $\gamma>0$ is taken such that $4\gamma L_2\leq 1$, we see that $\gamma a=3\gamma L_2\leq 3/4$.  Then by the convexity of $\ell^{[s]}$ we have
\begin{multline*}
(1-\gamma a)\ell^{[s]}(u) = -\gamma a \ell^{[s]}(u_0) + (1-\gamma a)\ell^{[s]}(u) + \gamma a\ell^{[s]}(u_0)\\
\geq -\gamma a \ell^{[s]}(u_0)+ \ell^{[s]}\left(\gamma a u_0 + (1-\gamma a)u\right).\end{multline*}
Similarly, by the convexity of $P^{[s]}(\cdot\vert \delta)$,
\[(1-\gamma a) P^{[s]}(u\vert \delta) \geq -\gamma a P^{[s]}(u_0\vert \delta) + P^{[s]}\left(\gamma a u_0 + (1-\gamma a)u\vert \delta\right).\]
Using these last two inequalities, and the change of variable $(1-\gamma a)u + \gamma a u_0= w$, we conclude that
\begin{multline*}
\int_{\rset^s} e^{-(1-\gamma a)\ell^{[s]}(u) -(1-\gamma a)P^{[s]}(u\vert\delta)}\rmd u \\
\leq e^{\gamma a\left(\ell^{[s]}(u_0) + P^{[s]}(u_0\vert \delta)\right)} \left(1-\gamma a\right)^{-s} \int_{\rset^s} e^{-\ell^{[s]}(u) -P^{[s]}(u\vert\delta)}\rmd u.\end{multline*}
Setting $\mathcal{R}(z)\eqdef\max_{\delta\in\Delta}\inf_{u\in\rset^s}\left[\ell^{[s]}(u) + P^{[s]}(u\vert \delta)\right]$, and using the inequality $\log(1-3x)\geq -6x$, $x\in [0,1/4]$ we obtain,
\[\int_{\rset^s} e^{-(1-\gamma a_1)\bar\ell(u) -(1-\gamma a_2)\bar P(u)}\rmd u \leq e^{\gamma a\mathcal{R}(z)} e^{6d\gamma L_2}\int_{\rset^s} e^{-\bar\ell^{(s)}(u) -\bar P^{(s)}(u\vert\delta)}\rmd u.\]
It follows from this last inequality, (\ref{control:h_gamma}) and (\ref{eq:control:R2}) that
\begin{equation*}\label{last:eq:prop1}
\varrho_\gamma(z) \leq \frac{3\gamma}{2} \max_{\delta\in\Delta} c(\delta) + 3\gamma d(L_1 +2 L_2) + 3\gamma L_2\mathcal{R}(z),\end{equation*}
as claimed.
\begin{flushright} $\square$ \end{flushright}

\subsection{Proof of Corollary \ref{prop1}}\label{sec:proof:prop1}
We show that H\ref{A1}-H\ref{A2} hold and apply Theorem \ref{thm2}.

The function $\ell$ is clearly convex and $\nabla\ell(\theta)=-\frac{1}{\sigma^2}X'(z-X\theta)$. Hence H\ref{A1}(1) holds.  The elastic-net density in (\ref{prior:p}) is log-concave and continuous, which implies that $P(\cdot\vert\delta)$ is convex and lower semi-continuous for  any given $\delta$. Furthermore, For $\theta\in\rset^d_\delta$, $\textsf{sign}(\theta)$ is a subgradient of $x\mapsto \|x\|_1$ at $\theta$. Hence $g(\theta\vert\delta)\eqdef \frac{\alpha\lambda_1}{\sigma^2}\textsf{sign}(\theta) + \frac{(1-\alpha)\lambda_2}{\sigma^2}\theta$ is a subgradient of $P(\cdot\vert\delta)$ at $\theta\in\rset^d_\delta$. Hence H\ref{A1} holds.

From the expression of $\nabla\ell$, we have
\[\|\nabla\ell(\theta_2)-\nabla\ell(\theta_1)\|\leq L_1\|\theta-\theta_2\|.\]
with $L_1 \eqdef \lambda_{\textsf{max}}(X'X)/\sigma^2$. Furthermore, for all $\delta\in\Delta$ and $\theta\in\rset^d_\delta$,
\begin{equation*}
\|\delta\cdot\nabla\ell(\theta)\|^2 = \frac{1}{\sigma^4}(z-X\theta)'X_\delta X_\delta'(z-X\theta)\leq 2L_1\frac{1}{2\sigma^2}\|z-X\theta\|^2.\end{equation*}
From the expression of $g(\cdot\vert\delta)$, we have
\begin{eqnarray*}
\|g(\theta\vert \delta)\|^2 &\leq& \left(\frac{\alpha\lambda_1}{\sigma^2}\right)^2\|\delta\|_0 + \frac{2(1-\alpha)\lambda_2}{\sigma^2}\left[\alpha\frac{\lambda_1}{\sigma^2}\|\theta\|_1 + (1-\alpha)\frac{\lambda_2}{2\sigma^2}\|\theta\|^2\right] \\
&=& c(\delta) + \frac{2(1-\alpha)\lambda_2}{\sigma^2} P(\theta\vert \delta),\\
&\leq & c(\delta) + 2L_1 P(\theta\vert \delta),\;\theta\in\rset^d_\delta
\end{eqnarray*}
where $c(\delta)\eqdef \left(\frac{\alpha\lambda_1}{\sigma^2}\right)^2\|\delta\|_0$, and using the assumption $(1-\alpha)\lambda_2\leq \lambda_{\textsf{max}}(X'X)$. These inequalities show that H\ref{A2} holds. The corollary then follows from Theorem \ref{thm2} by noting that $\max_\delta c(\delta)\leq \left(\frac{\alpha\lambda_1}{\sigma^2}\right)^2d$, and $\mathcal{R}(z) \leq \ell(0)=\frac{\|z\|^2}{2\sigma^2}$.
\begin{flushright}
$\square$
\end{flushright}
\vspace{1cm}

\subsection{Proof of Theorem \ref{thm3}}\label{sec:proof:thm3}
Since $\epsilon \eqdef Z-X\theta_\star\sim \textbf{N}(0,\sigma^2I_n)$, by standard Gaussian tail bound and union bound inequalities, we have
\begin{equation}\label{bound:prob:good:set}
\PP\left[Z\notin \mathcal{E}\right] = \PP\left[\max_{1\leq k\leq d}|\pscal{X_k}{\epsilon}|> \frac{\lambda_1}{2}\right] \leq 2\exp\left(\log(d) -\frac{\lambda_1^2}{8n\sigma^2\bar\kappa(1)}\right)= \frac{2}{d},\end{equation}
given the choice $\lambda_1 = 4\sigma\sqrt{\bar\kappa(1)n\log(d)}$. By Jensen's inequality,
\begin{eqnarray*}
\PE\left[\varrho_\gamma(Z)\vert Z\in\mathcal{E}\right] & \leq & \log \PE\left[e^{\varrho_\gamma(Z)}\vert Z\in\mathcal{E}\right]\\
&\leq & -\log\left(1-\frac{2}{d}\right) +\log\PE\left[e^{\varrho_\gamma(Z)}\textbf{1}_{\mathcal{E}}(Z)\right],\;\;\mbox{ on }\;\;Z\in\mathcal{E}.\end{eqnarray*}
In the particular case of the linear model, we have
\[e^{\varrho_\gamma(Z)} =\frac{\sum_{\delta}\pi_\delta\left(\frac{\lambda_1}{2\sigma^2}\right)^{\|\delta\|_0}\int_{\rset^{d}}e^{r_\gamma(\delta,\theta)}\left(\frac{1}{2\pi}\right)^{\frac{d-\|\delta\|_0}{2}}e^{-\frac{1}{2\gamma}\|\theta-\theta_\delta\|^2} e^{-\frac{1}{2\sigma^2}\|Z-X\theta_\delta\|^2} e^{-\frac{\lambda_1}{\sigma^2}\|\theta_\delta\|_1} \rmd \theta}{\sum_{\delta}\pi_\delta\left(\frac{\lambda_1}{2\sigma^2}\right)^{\|\delta\|_0}\int_{\rset^{\|\delta\|_0}} e^{-\frac{1}{2\sigma^2}\|Z-X_\delta u\|^2} e^{-\frac{\lambda_1}{\sigma^2}\|u\|_1} \rmd u}\]
Using Lemma 17 of \cite{atchade:15:b}, the denominator of $e^{\varrho_\gamma(Z)}$ satisfies the lower bound
\begin{multline}\label{proof:thm3:bound:denom}
e^{\frac{1}{2\sigma^2}\|Z-X\theta_\star\|^2}e^{\frac{\lambda_1}{\sigma^2}\|\theta_\star\|_1} \sum_{\delta}\pi_\delta\left(\frac{\lambda_1}{2\sigma^2}\right)^{\|\delta\|_0}\int_{\rset^{\|\delta\|_0}} e^{-\frac{1}{2\sigma^2}\|Z-X_\delta u\|^2} e^{-\frac{\lambda_1}{\sigma^2}\|u\|_1} \rmd u \\
\geq \pi_{\delta_\star} \left(1 + \frac{n\sigma^2\bar\kappa(s_\star)}{\lambda_1^2}\right)^{-s_\star}.\end{multline}
As in the proof of Theorem \ref{thm1}, and setting $M_\gamma \eqdef \frac{1}{\sigma^2}X'X\left(I_d +\frac{3\gamma}{2\sigma^2}X'X\right)$, we have
\begin{eqnarray*}
r_\delta(\delta,\theta) &\leq& \pscal{\nabla\ell(\theta)-\nabla\ell(\theta_\delta)}{\theta-\theta_\delta} +\frac{3 \gamma}{2} \|\nabla\ell(\theta)-\nabla\ell(\theta_\delta)\|^2  \nonumber\\
&&+ \frac{3 \gamma}{2} \|\delta\cdot\nabla\ell(\theta_\delta)\|^2 + \frac{3 \gamma}{2}\|g(\theta_\delta\vert\delta,\phi)\|^2\nonumber\\
& \leq  & (\theta-\theta_\delta)' M_\gamma  (\theta-\theta_\delta) \nonumber\\
&&+ \frac{3 \gamma}{2\sigma^4}\lambda_{\textsf{max}}(X_\delta X_\delta')\|Z-X\theta_\delta\|^2 + \frac{3\gamma}{2}\left(\frac{\lambda_1}{\sigma^2}\right)^2\|\delta\|_0.
\end{eqnarray*}
Set $a_\delta\eqdef (3/\sigma^2)\lambda_{\textsf{max}}(X_\delta X_\delta')$, and $H_\gamma \eqdef I -2\gamma M_\gamma$. Using the last inequality, we get the bound
\begin{multline}\label{proof:thm3:bound:num:1}
\int_{\rset^{d}}e^{r_\gamma(\delta,\theta)}e^{-\frac{1}{2\gamma}\|\theta-\theta_\delta\|^2} e^{-\frac{1}{2\sigma^2}\|Z-X\theta_\delta\|^2} e^{-\frac{\lambda_1}{\sigma^2}\|\theta_\delta\|_1} \rmd \theta\\
\leq e^{\frac{3\gamma}{2}\left(\frac{\lambda_1}{\sigma^2}\right)^2\|\delta\|_0}\int_{\rset^d} e^{-\frac{1}{2\gamma}(\theta-\theta_\delta)'H_\gamma(\theta-\theta_\delta)}e^{-\frac{1-\gamma a_\delta}{2\sigma^2}\|Z-X\theta_\delta\|^2-\frac{\lambda_1}{\sigma^2}\|\theta_\delta\|_1}\rmd \theta. \end{multline}
And if we call $J_\delta(Z)$ the integral on the right-side of (\ref{proof:thm3:bound:num:1}), then by Fubini's theorem,
\begin{multline}\label{proof:thm3:bound:num:2}
\PE\left[\textbf{1}_{\mathcal{E}}(Z)e^{\frac{1}{2\sigma^2}\|Z-X\theta_\star\|^2}e^{\frac{\lambda_1}{\sigma^2}\|\theta_\star\|_1}J_\delta(Z)\right]= \int_{\rset^d}e^{-\frac{1}{2\gamma}(\theta-\theta_\delta)'H_\gamma(\theta-\theta_\delta)} e^{-\frac{\lambda_1}{\sigma^2}\left[\|\theta_\delta\|_1-\|\theta_\star\|_1\right]}\\
\times \PE\left[\textbf{1}_{\mathcal{E}}(Z)\exp\left(\frac{1}{2\sigma^2}\|Z-X\theta_\star\|^2 -\frac{1-\gamma a_\delta}{2\sigma^2}\|Z-X\theta_\delta\|^2\right)\right]\rmd\theta.
\end{multline}
We write
\begin{multline*}
\frac{1}{2\sigma^2}\|Z-X\theta_\delta\|^2 = \frac{1}{2\sigma^2}\|Z-X\theta_\star\|^2 +\pscal{\frac{1}{\sigma^2}X'(Z-X\theta_\star)}{\theta_\delta-\theta_\star} \\
+\frac{1}{2\sigma^2}(\theta_\delta-\theta_\star)'(X'X)(\theta_\delta-\theta_\star),\end{multline*}
and for $Z\in\mathcal{E}$, $|\pscal{\frac{1}{\sigma^2}X'(Z-X\theta_\star)}{\theta_\delta-\theta_\star}|\leq (\lambda_1/2\sigma^2)\|\theta_\delta-\theta_\star\|_1$. Therefore the expectation on the right-side of (\ref{proof:thm3:bound:num:2}) is upper bounded by
\begin{multline*}
e^{\frac{\lambda_1}{2\sigma^2}\|\theta_\delta-\theta_\star\|_1} e^{-\frac{1-\gamma a_\delta}{2\sigma^2}(\theta_\delta-\theta_\star)'(X'X)(\theta_\delta-\theta_\star)}\PE\left[e^{\frac{\gamma a_\delta}{2\sigma^2}\|Z-X\theta_\star\|^2}\right]\\
= e^{\frac{\lambda_1}{2\sigma^2}\|\theta_\delta-\theta_\star\|_1} e^{-\frac{1-\gamma a_\delta}{2\sigma^2}(\theta_\delta-\theta_\star)'(X'X)(\theta_\delta-\theta_\star)}\left(\frac{1}{1-\gamma a_\delta}\right)^{n/2}.
\end{multline*}
Letting
\[B(\theta)\eqdef -\frac{\lambda_1}{\sigma^2}\left[\|\theta\|_1-\|\theta_\star\|_1\right] + \frac{\lambda_1}{2\sigma^2}\|\theta-\theta_\star\|_1 - \frac{1-\gamma a_\delta}{2\sigma^2}(\theta-\theta_\star)'(X'X)(\theta-\theta_\star),\]
for $\theta\in\rset^d$, we conclude that
\begin{multline}\label{proof:thm3:bound:num:3}
\PE\left[\textbf{1}_{\mathcal{E}}(Z)e^{\frac{1}{2\sigma^2}\|Z-X\theta_\star\|^2}e^{\frac{\lambda_1}{\sigma^2}\|\theta_\star\|_1}J_\delta(Z)\right]\\
\leq  \left(\frac{1}{1-\gamma a_\delta}\right)^{n/2} \int_{\rset^d}e^{-\frac{1}{2\gamma}(\theta-\theta_\delta)'H_\gamma(\theta-\theta_\delta)} e^{B(\theta_\delta)}\rmd\theta.
\end{multline}
Using an argument that can be found in \cite{castillo:etal:14}~(proof of Theorem 10), and also in \cite{atchade:15:b}~(proof of Lemma 5), it can be shown that the function $B$ satisfies
\begin{equation}\label{bound:B}
B(\theta)\leq \frac{8\lambda_1^2 s_\star}{n\sigma^2\underline{\kappa}} -\frac{\lambda_1}{4\sigma^2}\|\theta-\theta_\star\|_1,\;\;\theta\in\rset^d.\end{equation}
Combining (\ref{bound:B}), (\ref{proof:thm3:bound:num:3}), (\ref{proof:thm3:bound:num:1}), and (\ref{proof:thm3:bound:denom}), we conclude that
\begin{multline}\label{proof:thm3:bound:rho:1}
\PE\left[\textbf{1}_{\mathcal{E}}(Z) e^{\varrho_\gamma(Z)}\right]\leq \frac{1}{\pi_{\delta_\star}}\left(1 + \frac{n\sigma^2\bar\kappa(s_\star)}{\lambda_1^2}\right)^{s_\star} e^{\frac{8\lambda_1^2s_\star}{n\sigma^2\underline{\kappa}}}\sum_{\delta} \pi_\delta  e^{\frac{3\gamma}{2}\left(\frac{\lambda_1}{\sigma^2}\right)^2\|\delta\|_0} e^{-\frac{n}{2}\log(1-\gamma a_\delta)} \\
\times \left(\frac{\lambda_1}{2\sigma^2}\right)^{\|\delta\|_0}\left(\frac{1}{2\pi}\right)^{\frac{d-\|\delta\|_0}{2}}\int_{\rset^d}e^{-\frac{1}{2\gamma}(\theta-\theta_\delta)'H_\gamma(\theta-\theta_\delta)} e^{-\frac{\lambda_1}{4\sigma^2}\|\theta_\delta-\theta_\star\|_1}\rmd\theta.
\end{multline}
Since $(\gamma/\sigma^2)\lambda_{\textsf{max}}(X'X)\leq 1/4$, and using the inequality $\log(1-2x-3x^2)\geq -6x$, valid for all $x\in [0,1/4]$, it can be shown that the integral on the right-side of (\ref{proof:thm3:bound:rho:1}) is upper bound by
\[ \left(\frac{8\sigma^2}{\lambda_1}\right)^{\|\delta\|_0} (2\pi)^{\frac{d-\|\delta\|_0}{2}}e^{\frac{3\gamma}{\sigma^2}\textsf{Tr}(X'X)}.\]
We conclude that
\begin{multline*}
\PE\left[e^{\varrho_\gamma(Z)}\textbf{1}_{\mathcal{E}}(Z)\right]\leq \frac{1}{\pi_{\delta_\star}}\left(1 + \frac{n\sigma^2\bar\kappa(s_\star)}{\lambda_1^2}\right)^{s_\star} e^{\frac{8\lambda_1^2s_\star}{n\sigma^2\underline{\kappa}}}e^{\frac{3\gamma}{\sigma^2}\textsf{Tr}(X'X)} \\
\sum_{\delta}\pi_\delta  e^{\frac{3\gamma}{2}\left(\frac{\lambda_1}{\sigma^2}\right)^2\|\delta\|_0} e^{-\frac{n}{2}\log(1-\gamma a_\delta)} e^{\log(4)\|\delta\|_0}.
\end{multline*}
Since $-\frac{n}{2}\log(1-\gamma a_\delta)\leq 3\gamma n\lambda_{\textsf{max}}(X'X)/\sigma^2$, and with $\lambda_1=4\sigma\sqrt{\bar\kappa(1)n\log(d)}$,
\begin{multline*}
\log \PE\left[e^{\varrho_\gamma(Z)}\textbf{1}_{\mathcal{E}}(Z)\right] \leq  -\log(\pi_{\delta_\star}) +s_\star\log\left(1+\frac{\bar\kappa(s_\star)}{16\log(d)}\right) + 128\left(\frac{\bar\kappa(1)}{\underline{\kappa}}\right)\log(d)\\
+\frac{3\gamma}{\sigma^2}\left(n\lambda_{\textsf{max}}(X'X) + \textsf{Tr}(X'X)\right) +\log\sum_{\delta}\pi_\delta e^{\frac{3\gamma}{2}\left(\frac{\lambda_1}{\sigma^2}\right)^2\|\delta\|_0} e^{\log(4)\|\delta\|_0}.
\end{multline*}
Set $A \eqdef \log(4) + \frac{3\gamma}{2}\left(\frac{\lambda_1}{\sigma^2}\right)^2$. Then
\begin{eqnarray*}
\log\sum_\delta \pi_\delta e^{A\|\delta\|_0} &=& \log\sum_{s=0}^d {d \choose s}\textsf{q}^s(1-\textsf{q})^{d-s}e^{As}\\
& = & d\log\left(1+\textsf{q}e^A(1-e^{-A})\right),\\
& \leq & 4\log(4) + 6\gamma\left(\frac{\lambda_1}{\sigma^2}\right)^2,
\end{eqnarray*}
for $\frac{24\gamma}{\sigma^2}\bar\kappa(1)\leq u-1$. Also,
\[-\log(\pi_{\delta_\star}) \leq us_\star\log(d) + \frac{2}{d^{u-1}}.\]
The theorem is proved.
\begin{flushright}
$\square$
\end{flushright}

\vspace{2cm}

\bibliographystyle{ims}

\end{document}

%% file: header_new.tex
\usepackage{amsmath,amssymb,amsfonts,amsthm,enumerate,natbib,color,ifthen,hyperref}
\usepackage{xargs}
\usepackage[textwidth=4cm, textsize=footnotesize]{todonotes}
\usepackage[latin1]{inputenc}

\usepackage{float}
\usepackage[caption = false]{subfig}
\usepackage{graphicx}

\usepackage{aliascnt,bbm}
\def\rset{\mathbb R}
\def\zset{\mathbb Z}

\def\eqsp{\;}

 \newcommand{\pscal}[2]{\left\langle#1,#2\right\rangle}

\newcommand{\eqdef}{\ensuremath{\stackrel{\mathrm{def}}{=}}}

\def\Xset{\mathsf{X}} 

\def\dist{\textsf{d}}

\newcommandx\sequence[3][2=t,3=\zset]
{\ifthenelse{\equal{#3}{}}{\ensuremath{\{ #1_{#2}\}}}{\ensuremath{\{ #1_{#2}, \eqsp #2 \in #3 \}}}}

\def\PP{\mathbb{P}} 
\newcommand{\CPP}[3][]
{\ifthenelse{\equal{#1}{}}{{\mathbb P}\left(\left. #2 \, \right| #3 \right)}{{\mathbb P}_{#1}\left(\left. #2 \, \right | #3 \right)}}
\def\PE{\mathbb{E}} 
\newcommand{\CPE}[3][]
{\ifthenelse{\equal{#1}{}}{{\mathbb E}\left[\left. #2 \, \right| #3 \right]}{{\mathbb E}_{#1}\left[\left. #2 \, \right | #3 \right]}}

\def\tv{\mathrm{tv}}

\def\piq{\textsf{q}}

\def\c{\textsf{c}}

\def\Prox{\operatorname{Prox}}

\usepackage{bm}

\theoremstyle{plain}
\newtheorem{theorem}{Theorem}
\newtheorem{assumption}{H\hspace{-3pt}}

\newaliascnt{proposition}{theorem}
\newtheorem{proposition}[proposition]{Proposition}
\aliascntresetthe{proposition}

\newaliascnt{lemma}{theorem}
\newtheorem{lemma}[lemma]{Lemma}
\aliascntresetthe{lemma}
\newaliascnt{corollary}{theorem}
\newtheorem{corollary}[corollary]{Corollary}
\aliascntresetthe{corollary}

\theoremstyle{definition}
\newaliascnt{definition}{theorem}

\aliascntresetthe{definition}

\newaliascnt{remark}{theorem}
\newtheorem{remark}[remark]{Remark}
\aliascntresetthe{remark}

\newaliascnt{example}{theorem}
\newtheorem{example}[example]{Example}
\aliascntresetthe{example}

\def\rmd{\mathrm{d}}

\def\1{\mathbbm{1}}

